\documentclass[12pt]{amsart}

\usepackage{amsfonts}
\usepackage{amsmath}
\usepackage{amsrefs}
\usepackage{amssymb}
\usepackage{hyperref}
\usepackage{url}
\usepackage[shortlabels]{enumitem}
\usepackage[paper=a4paper,left=20mm,right=20mm,top=25mm,bottom=25mm]{geometry}

\usepackage{tikz}
\usetikzlibrary{mindmap}
\usetikzlibrary{backgrounds}
\usetikzlibrary{trees,snakes}
\usepackage[outline]{contour}
\contourlength{1.5pt}
\usetikzlibrary{positioning}
\usetikzlibrary{%
  matrix,%
  calc,%
  arrows%
}

\let\emptyset\varnothing
\nocite{*}

\newcommand{\A}{\mathbb{A}}

\newcommand{\NN}{\mathbb{N}}

\newcommand{\ZZ}{\mathbb{Z}}

\newcommand{\cC}{\mathcal{C}}
\newcommand{\cD}{\mathcal{D}}

\newcommand{\cE}{\mathcal{E}}
\newcommand{\cF}{\mathcal{F}}

\newcommand{\cK}{\mathcal{K}}

\newcommand{\cW}{\mathcal{W}}
\newcommand{\cX}{\mathcal{X}}

\newcommand{\cI}{\mathcal{I}}
\newcommand{\cP}{\mathcal{P}}
\newcommand{\cS}{\mathcal{S}}
\newcommand{\cR}{\mathcal{R}}

\DeclareMathOperator{\Aut}{Aut}

\DeclareMathOperator{\rad}{rad}
\DeclareMathOperator{\rep}{rep}
\DeclareMathOperator{\ggt}{gcd}
\DeclareMathOperator{\ind}{ind}

\DeclareMathOperator{\ql}{ql}

\DeclareMathOperator{\Hom}{Hom}

\DeclareMathOperator{\rk}{rk}

\DeclareMathOperator{\modd}{mod}
\DeclareMathOperator{\supp}{supp}

\DeclareMathOperator{\Ext}{Ext}

\DeclareMathOperator{\id}{id}

\DeclareMathOperator{\GL}{GL}

\DeclareMathOperator{\EIP}{EIP}

\DeclareMathOperator{\Inj}{Inj}
\DeclareMathOperator{\Sur}{Surj}

\DeclareMathOperator{\EKP}{EKP}

\DeclareMathOperator{\dimu}{\underline{dim}}

\let\emptyset\varnothing
\newtheorem{proposition}{Proposition}[subsection]
\newtheorem{Theorem}[proposition]{Theorem}
\newtheorem{Lemma}[proposition]{Lemma}

\newtheorem{corollary}[proposition]{Corollary}
\newtheorem*{TheoremN}{Theorem}
\newtheorem*{corollaryN}{Corollary}
\newtheorem*{PropositionN}{Proposition}
\newtheorem{TheoremS}{Theorem}[section]
\newtheorem{corollaryS}[TheoremS]{Corollary}
\newtheorem{LemmaS}[TheoremS]{Lemma}

\newenvironment{examples}[1][Examples.]{\begin{trivlist}
\item[\hskip \labelsep {\bfseries #1}]}{\end{trivlist}}
\newenvironment{Remark}[1][Remark.]{\begin{trivlist}
\item[\hskip \labelsep {\bfseries #1}]}{\end{trivlist}}

\newenvironment{Definition}[1][Definition.]{\begin{trivlist}
\item[\hskip \labelsep {\bfseries #1}]}{\end{trivlist}}

\begin{document}

\title[Invariants of AR-components for generalized Kronecker quivers]{Representations of regular trees and invariants of AR-components for generalized Kronecker quivers}
\author{Daniel Bissinger}
\address{Christian-Albrechts-Universit\"at zu Kiel, Ludewig-Meyn-Str. 4, 24098 Kiel, Germany}
\email{bissinger@math.uni-kiel.de}
\date{}
\thanks{Partly supported by the D.F.G. priority program SPP 1388  ``Darstellungstheorie''}
\maketitle
%\markright{REPRESENTATIONS OF REGULAR TREES AND INVARIANTS OF AR-COMPONENTS FOR GENERALIZED KRONECKER QUIVERS}

%%%%%%%%%%%%%%%%%%%%%%%% Subjects classification %%%%%%%%%%%%%%%%%%%
\makeatletter
\def\blfootnote{\gdef\@thefnmark{}\@footnotetext}
\makeatother

\blfootnote{\textup{2010} \textit{Mathematics Subject Classification}: 16G20, 16G60}

%%%%%%%%%%%%%%%%%%%%%%%%%%%%%%%%%% Abstract %%%%%%%%%%%%%%%%%%%%%%%%

\begin{abstract}
We investigate the generalized Kronecker algebra $\cK_r = k\Gamma_r$ with $r \geq 3$ arrows. Given a regular component $\cC$ of the Auslander-Reiten quiver of $\cK_r$, we show that the quasi-rank $\rk(\cC) \in \ZZ_{\leq 1}$ can be described almost exactly as the distance $\cW(\cC) \in \NN_0$ between two non-intersecting cones in $\cC$, given by modules with the equal images and the equal kernels property; more precisley, we show that the two numbers are linked by the inequality 
\[ -\cW(\cC) \leq \rk(\cC) \leq - \cW(\cC) + 3.\]
Utilizing covering theory, we construct for each $n \in \NN_0$ a bijection $\varphi_n$ between the field $k$ and $\{ \cC \mid \cC \ \text{regular component}, \ \cW(\cC) = n \}$. As a consequence, we get new results about the number of regular components of a fixed quasi-rank.

\end{abstract}

\section*{Introduction}
Let $k$ be an algebraically closed field of arbitrary characteristic. The finite-dimensional algebra $\cK_r$ is defined as the path algebra of the quiver $\Gamma_r$ with vertices $1,2$, $r \in \NN$ arrows $\gamma_1,\ldots,\gamma_r \colon 1 \to 2$ and called the generalized Kronecker algebra. We denote by $\modd \cK_r$ the category of finite-dimensional left-modules of $\cK_r$.\\
It is well known that for $r \geq 3$ the hereditary algebra $\cK_r$ is of wild representation type \cite[1.3,1.5]{Ker3}, every regular component in the  Auslander-Reiten quiver of $\cK_r$ is of type $\ZZ A_\infty$ \cite{Ri3} and there is a bijection between the regular components and the ground field $k$ \cite[XVIII 1.8]{Assem3}. Therefore, the problem of completely understanding the category $\modd \cK_r$ or all regular components is considered hopeless and it is desirable to find invariants which give more specific information about the regular components. \\
One important invariant $($for any  wild hereditary algebra$)$, introduced in \cite{Ker1}, is the quasi-rank $\rk(\cC) \in \ZZ$  of a regular component $\cC$. For a quasi-simple module $X$ in $\cC$, $\rk(\cC)$ is defined as  
\[ \rk(\cC) := \min \{ l \in \ZZ \mid \forall m \geq l \colon \rad(X,\tau^{m} X) \neq 0 \},\]
where $\rad(X,Y)$ is the space of non-invertible homomorphisms from $X$ to $Y$.\\
Another interesting invariant $\cW(\cC) \in \NN_0$ was recently introduced in \cite{Wor1}. Motivated by the representation theory of group algebras of $p$-elementary abelian groups of characteristic $p > 0$, the author defines the category $\EKP$ of modules with the equal kernels property and the category $\EIP$ of modules with the equal images property in the framework of $\cK_r$. 
She shows the existence of uniquely determined quasi-simple modules $M_\cC$ and $W_\cC$ in $\cC$ such that 
the cone $(M_\cC \rightarrow)$ of all successors of $M_\cC$ satisfies $(M_\cC \rightarrow) = \EKP \cap \cC$ and the cone $(\rightarrow W_\cC)$ of all predecessors of $W_\cC$ satisfies
$(\rightarrow W_\cC)= \EIP \cap \cC$. The width $\cW(\cC)$ of $\cC$ is defined as the unique number $\cW(\cC) \in \NN_0$ such that $\tau^{\cW(\cC)+1} M_\cC = W_\cC$, i.e. the distance between the two cones.\\
Utilizing homological descriptions of $\EKP$ and $\EIP$ from \cite{Wor1} involving a family of elementary modules, we show that the two invariants $\rk(\cC)$ and $\cW(\cC)$ are linked by the inequality 
\[  \ -\cW(\cC) \leq \rk(\cC) \leq - \cW(\cC) + 3. \]
Motivated by this connection, we construct for each $n \in \NN$ a regular component $\cC$ with $\cW(\cC) = n$. In order to do so, we consider representations over the universal covering $C_r$ of $\Gamma_r$.

We define classes $\Inj$,$\Sur$ of representations over $C_r$ such that $M \in \rep(C_r)$ is in $\Inj$ $($resp. $\Sur)$ if and only if for each arrow $\delta \colon x \to y$ of $C_r$ the linear map $M(\delta) \colon M_x \to M_y$ is injective $($resp. surjective$)$. \\
Let $\pi_{\lambda} \colon \rep(C_r) \to \rep(\Gamma_r)$ be the push-down functor \cite[2.7]{Gab3} and $M \in \rep(C_r)$ be indecomposable. We prove that $M$ is in $\Inj$ $($resp. $\Sur)$ if and only if $\pi_{\lambda}(M)$ is in $\EKP$ $($resp. $\EIP)$. Since a component $\cD$ of the Auslander-Reiten quiver of $C_r$ which is taken to a regular component $\cC := \pi_{\lambda}(\cD)$ is also of type $\ZZ A_\infty$, we can lift the definition of $\cW(\cC)$ to $\cD$. We define $\cW_C(\cD) \in \NN_0$ as the distance between the cones $\Sur \cap \cD$ and $\Inj \cap \cD$ and show that $\cW_{C}(\cD) = \cW(\cC)$.\\
For $X \in \cD$, we denote by $\ql(X)$ its quasi-length. If $X$ has certain properties, we show how to construct a short exact sequence $\delta_X = 0 \to Y \to E \to X \to 0$ with indecomposable middle term $E$ in a component $\cE$ such that \[ (\ast) \quad \cW_C(\cE) = \cW_C(\cD) + \ql(X) - 1.\] 
The construction of $\delta_X$ relies on the fact that $C_r$ is an infinite $r$-regular tree with bipartite orientation. Using $(\ast)$, we construct for each $n \in \NN$ a component $\cD_n$ with $\cW_C(\cD_n) = n$.\\
In conjunction with a natural action of $\GL_r(k)$ on $\rep(\Gamma_r)$, we arrive at our main theorem:

\begin{TheoremN}
Let $n \in \NN_0$. There is a bijection $k \to \{ \cC  \mid \cC \ \text{regular component of} \ \Gamma_r, \cW(\cC) = n\}$.
\end{TheoremN}

\noindent As an immediate consequence we get the following statements, which are generalizations of results by Kerner and Lukas \cite[3.1]{KerLuk2}, \cite[5.2]{KerLuk2} for the Kronecker algebra.

\begin{corollaryN}
Let $r \geq 3$, then for each $n \in \NN$ there are exactly $|k|$ regular components with quasi-rank in $\{-n,-n+1,-n+2,-n+3\}$.
\end{corollaryN}

\begin{corollaryN} Assume that $k$ is uncountable and $q \in \NN$. The set of components of quasi-rank $\leq -q$ is uncountable.
\end{corollaryN}

\section{Preliminaries and Motivation}

\subsection{Notations and basic results}

Throughout this article let $k$ be an algebraically closed field of arbitrary characteristic. 
For a quiver $Q = (Q_0,Q_1,s,t)$, $x \in Q_0$ let
\[x^+ := \{ y \in Q_0 \mid \exists \alpha \colon x \to y\} \ \text{and} \ x^- := \{ y \in Q_0 \mid \exists \alpha \colon y \to x\}. \]  
Moreover let $n(x) := x^+ \cup x^-$. If $\alpha \colon x \to y$, then by definition $s(\alpha) = x$ and $t(\alpha) = y$. We say that $\alpha$ starts in $s(\alpha)$ and ends in $t(\alpha)$. 

\begin{Definition}
A quiver $Q$ is called 
\begin{enumerate}[topsep=0em, itemsep= -0em, parsep = 0 em, label=$(\alph*)$]
\item locally finite if $n(x)$ is finite for each $x \in Q_0$,
\item of bounded length if for each $x \in Q_0$ there is $N_x \in \NN$ such that each path in $Q$ which starts or ends in $x$ is of length $\leq N_x$,
\item locally bounded if $Q$ is locally finite and of bounded length.
\end{enumerate}
\end{Definition}

From now on we assume that $Q$ is locally bounded. Note that this implies that $Q$ is acyclic. Moreover, we assume that $Q$ is connected.
A finite dimensional representation $M = ((M_x)_{x \in Q_0},(M(\alpha))_{\alpha \in Q_1})$  over $Q$ consists of vector spaces $M_x$ and linear maps $M(\alpha) \colon M_{s(\alpha)} \to M_{t(\alpha)}$ such that $\dim_k M := \sum_{x \in Q_0} \dim_k M_x$ is finite. A morphism $f \colon M \to N$ between representations is a collection of linear maps $(f_z)_{z \in Q_0}$ such for each arrow $\alpha \colon x \to y$ there is a commutative diagram

\[	   \begin{tikzpicture}[descr/.style={fill=white,inner sep=1.5pt}]
		
				\matrix [
            matrix of math nodes,
            row sep=3em,
            column sep=3.0em,
            text height=2.0ex, text depth=0.25ex
        ] (s)
{
& |[name=A_1]| M_x  &|[name=B_1]| M_y  \\
& |[name=A_2]| N_x &|[name=B_2]| N_y.  \\
};
\draw[->] (A_1) edge node[auto] {$M(\alpha)$} (B_1)         			 	  (A_2) edge node[auto] {$N(\alpha)$} (B_2)   
		  (A_1) edge node[auto] {$f_x$} (A_2)	
		  (B_1) edge node[auto] {$f_y$} (B_2)					      
				  ;
\end{tikzpicture}\]

\noindent The category of finite dimensional representations over $Q$ is denoted by $\rep(Q)$. The path category $k(Q)$ has $Q_0$ as set of objects and $\Hom_{k(Q)}(x,y)$ is the vector space with basis given by the paths from $x$ to $y$. The trivial arrow in $x$ is denoted by $\epsilon_x$. Since $Q$ is locally bounded, the category $k(Q)$ is locally bounded in the sense of \cite[2.1]{Gab2}. A finite-dimensional module over a locally bounded category $A$ is a functor $F \colon  A \to \modd k$ such that $\sum_{x \in A} \dim_k F(x)$ is finite. The category $\modd A$ has Auslander-Reiten sequences $($see \cite[2.2]{Gab2}$)$. Since a finite dimensional module over $k(Q)$ is the same as a representation of $Q$, the category $\rep(Q)$ has Auslander-Reiten sequences.\\
If moreover $Q_0$ is a finite set, we denote with $kQ$ the path algebra of $Q$ with idempotents $e_x, x \in Q_0$. In this case $kQ$ is a finite dimensional, associative, basic and connected $k$-algebra. We denote by $\modd kQ$ the class of finite-dimensional $kQ$ left modules. Given $M \in \modd kQ$ we let $M_x := e_x M$. The categories $\modd kQ$ and $\rep(Q)$ are equivalent $($see for example \cite[III 1.6]{Assem1}$)$. We will therefore switch freely between representations of $Q$ and modules of $kQ$, if one of the approaches seems more convenient for us. 
 We assume that the reader is familiar with Auslander-Reiten theory and basic results on wild hereditary algebras. For a well written survey on the subjects we refer to \cite{Assem1}, \cite{Ker2} and \cite{Ker3}. 
Recall the definition of the dimension function
\[ \dimu \colon \modd kQ \to \ZZ^{Q_0}, M \mapsto (\dim_k M_x)_{x \in Q_0}.\] 
If $0 \to A \to B \to C \to 0$ is an exact sequence, then $\dimu A + \dimu C = \dimu B$. A finite quiver $Q$ defines a $($non-symmetric) bilinear form 
$\langle -,- \rangle \colon \ZZ^{Q_0} \times \ZZ^{Q_0} \to \ZZ, $
given by $((x_i),(y_j)) \mapsto \sum_{i \in Q_0} x_i y_i - \sum_{\alpha \in Q_1} x_{s(\alpha)} y_{t(\alpha)}$, which coincides with the Euler-Ringel form \cite{Ri4} on the Grothendieck group $K_0(kQ) \cong \ZZ^{Q_0}$, i.e. for $M,N \in \modd kQ$ we have 
\[ \langle \dimu M, \dimu N \rangle = \dim_k \Hom(M,N) - \dim_k \Ext(M,N).\]

\subsection{The Kronecker algebra and $\ZZ A_\infty$ components}

\noindent We always assume that $r \geq 3$. Denote by $\Gamma_r$ the $r$-Kronecker quiver, which is given by two vertices $1,2$ and arrows $\gamma_1,\ldots,\gamma_r \colon 1 \to 2$. 

\begin{figure}[!h]
\centering 
\tikzstyle{every node}=[]
\tikzstyle{edge from child}=[]

\begin{tikzpicture}[->,>=stealth',auto,node distance=3cm,
  thick,main node/.style={circle,draw,font=\sffamily\Large\bfseries}]

  \node (1) {$\bullet$};
  \node (2) [right of=1] {$\bullet$};
   \node[color=black] at (1.4,0.7) {$\gamma_{1}$};
   \node[color=black] at (1.4,-0.7) {$\gamma_{r}$};
   \node[color=black] at (1.4,-0.2) {$\vdots$};
   \node[color=black] at (1.4,0.4) {$\vdots$};
   
   \node[color=black] at (0,-0.3) {$1$};
   \node[color=black] at (3,-0.3) {$2$};

  \path[every node/.style={font=\sffamily\small}]
    
    (1) edge[bend right] node [left] {} (2)
    (1) edge[bend left] node [left] {} (2)
    (1) edge[] node [left] {} (2)
   ;
\end{tikzpicture}
\caption{The Kronecker quiver $\Gamma_r$.}
\label{Fig:Kronecker}

\end{figure}
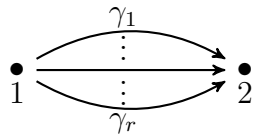

We set $\cK_r := k\Gamma_r$ and $P_1 := \cK_r e_2$, $P_2 := \cK_r e_1$. The modules $P_1$ and $P_2$ are the indecomposable projective modules of $\modd \cK_r$, $\dim_k \Hom(P_1,P_2) = r$ and $\dim_k \Hom(P_2,P_1) = 0$. As Figure $\ref{Fig:Kronecker}$ suggests, we write $\dimu M = (\dim_k M_1,\dim_k M_2)$. For example $\dimu P_1 = (0,1)$ and $\dimu P_2 = (1,r)$.

\begin{figure}[!h]
\centering 
\includegraphics[width=0.7\textwidth, height=50px]{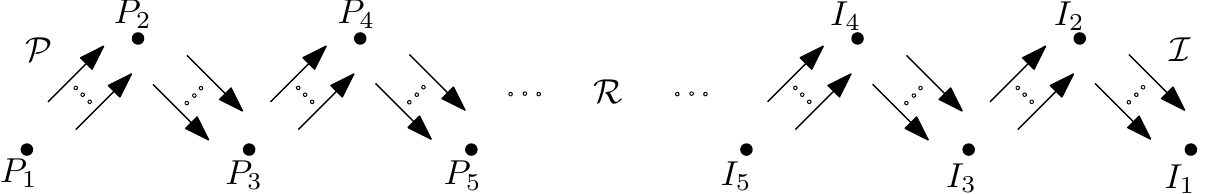}
\caption{AR quiver of $\Gamma_r$}
\label{Fig:AuslanderReitenquiver}
\end{figure}

Figure $\ref{Fig:AuslanderReitenquiver}$  shows the notation we use for the components $\cP,\cI$ in the Auslander-Reiten quiver of $\cK_r$ which contain the indecomposable projective modules $P_1,P_2$ and indecomposable injective modules $I_1,I_2$. The set of all other components is denoted by $\cR$.

Ringel has proven \cite[2.3]{Ri3} that every component in $\cR$ is of type $\ZZ A_\infty$. A module in such a component is called regular and the class of all regular indecomposable modules is denoted by $\ind \cR$. An irreducible morphism in a component of type $\ZZ A_\infty$ $($for any algebra$)$ is injective if the corresponding arrow is uprising and surjective otherwise. A representation $M$ in a $\ZZ A_\infty$ component is called quasi-simple if the AR sequence terminating in M has an indecomposable middle term. These modules are the modules in the bottom layer of the component. If $M$ is quasi-simple in a component $\cC$ of type $\ZZ A_\infty$, then there is an infinite chain $($ray$)$ of irreducible monomorphisms $($resp. epimorphisms$)$
\[ M = M[1] \to M[2] \to M[3] \to \cdots \to M[l] \to \cdots \] 
\[ \cdots (l)M \to \cdots \to (3)M \to (2)M \to (1)M = M, \] 
and for each indecomposable module $X$ in $\cC$ there are unique quasi-simple modules $N,M$ and $l \in \NN$ with $(l)M = X = N[l]$. The number $\ql(X) := l$ is called the quasi-length of $X$.\\    
The indecomposable modules in $\cP$ are called preprojective modules and the modules in $\cI$ are called preinjective modules. Moreover we call an arbitrary module preprojective $($resp. preinjective, regular$)$ if all its indecomposable direct summands are preprojective $($res. preinjective, regular$)$. We have $P$ in $\cP$ $(I$ in $\cI)$ if and only if there is $l \in \NN_0$ with $\tau^l P = P_i$ $ (\tau^{-l} I = I_i)$ for $i \in \{1,2\}$.
Let $\modd_{pf} \cK_r$ be the subcategory of all modules without non-zero projective direct summands and $\modd_{if} \cK_r$ the subcategory of all modules without non-zero injective summands. Since $\cK_r$ is a hereditary algebra, the Auslander-Reiten translation $\tau \colon \modd \cK_r \to \modd \cK_r$ induces an equivalence from $\modd_{pf} \cK_r$ to $\modd_{if} \cK_r$, that will be often used later on without further notice.

\subsection{The connection between $\rk(\cC)$ and $\cW(\cC)$}

Let us start by recalling the definitions of $\rk(\cC)$ and $\cW(\cC)$.

\begin{Definition}
Let $A = kQ$ be a wild hereditary algebra and $\cC$ be a regular component with a quasi-simple module $X$  in $\cC$, then 
\[ \rk(\cC) := \min \{ l \in \ZZ \mid \forall m \geq l \colon \rad(X,\tau^{m} X) \neq 0 \}.\]
\end{Definition}

It was shown in \cite{Ker1}, that $\overline{t}(A) := \max\{ \rk(\cC) \mid \cC \ \text{regular component} \}$ is finite and $\overline{t}(\cK_r) = 1$. In other words, there are lots of morphisms in $\tau$-direction. Also in $\tau^{-1}$-direction one finds a lot of morphisms since 
$ \underline{t}(A):=\inf \{ \rk(\cC) \mid \cC  \ \text{regular component} \} = -\infty $ $($see \cite[3.1]{KerLuk2}$)$.

Another invariant that can be attached to a regular component $\cC$ of the Kronecker algebra is defined in \cite{Wor1} and denoted by $\cW(\cC)$. In order to define $\cW(\cC)$, we first recall the definitions of $\EKP$ and $\EIP$. For $\alpha \in k^{r}\setminus \{0\}$ and $M \in \rep(\Gamma_r)$ we consider the $k$-linear map $M^\alpha := \sum^{r}_{i=1} \alpha_i M(\gamma_i) \colon M_1 \to M_2$ and let $X_\alpha \in \rep(\Gamma)$ be the cokernel of the embedding $(0,\iota_2) \colon P_1 \to P_2$ where $\iota_2 \colon k \to k^r$, $x \mapsto x\alpha$.

\begin{Definition}\cite[2.1]{Wor1} We define the classes of representations with the equal kernels property and with the equal images property as
\begin{enumerate}
\item $\EKP := \{ M \in \rep(\Gamma_r) \mid \forall \alpha \in k^{r}: M^\alpha \ \text{is injective} \}$ and  
\item $\EIP := \{ M \in \rep(\Gamma_r) \mid \forall \alpha \in k^{r}:  M^\alpha \ \text{is surjective}\}.$
\end{enumerate}
\end{Definition}

Given a regular component $\cC$, there exist uniquely determined quasi-simple representations $M_\cC$ and $W_\cC$ in $\cC$ such that 
$(M_\cC \to) = \EKP \cap \cC$ and $(\to W_\cC) = \EIP \cap \cC$ $($see \cite[3.3]{Wor1}$)$. Now $\cW(\cC)$ is defined as the unique integer with $\tau^{\cW(\cC) + 1} M_\cC= W_\cC$. Since $\EKP \cap \EIP = \{0\}$ it follows $\cW(\cC) \geq 0$.

\begin{proposition}
Let $\cC$ be a regular component, then $- \cW(\cC) \leq \rk(\cC) \leq -\cW(\cC) +3$.
\end{proposition}
\begin{proof}
On the one hand let $M :=  \tau^{-1} W_\cC$, then there is $\alpha \in k^r\setminus \{0\}$ with $0 \neq \Ext(X_\alpha,M)$ \cite[2.5]{Wor1}. The Auslander-Reiten formula \cite[2.3]{Ker3} yields $0 \neq \Hom(\tau^{-1} M,X_\alpha)$. On the other hand let $N := \tau M_\cC$, then there is $\beta \in k^r \setminus \{0\}$ with $0 \neq \Hom(X_\beta,N) \cong \Hom(\tau X_\beta,\tau N)$ \cite[2.5]{Wor1}. By the Euler-Ringel form we have
\[ 0 > 2 - r  = 1 + (r-1)^2 - r(r-1) = \langle \dimu X_\beta,\dimu X_\alpha \rangle = \dim_k \Hom(X_\beta,X_\alpha) - \dim_k \Ext(X_\beta,X_\alpha).\]
Hence $0 \neq \dim_k \Ext(X_\beta,X_\alpha) = \dim_k \Hom(\tau^{-1} X_\alpha,X_\beta) = \dim_k \Hom(X_\alpha,\tau X_\beta)$. Since $X_\alpha$ and $\tau X_\beta$ are elementary \cite[2.1.4]{Bi1}, we get a non-zero morphism  by \cite[2.1.1]{Bi1}
\[ \tau^{-1} M \to X_\alpha \to \tau X_\beta \to \tau N, \ \text{and} \]
\[0 \neq \Hom(M,\tau^2 N) = \Hom(\tau^{-1} W_\cC,\tau^2 \circ \tau^{-\cW(\cC)} W_\cC)= \Hom(W_\cC,\tau^{-\cW(\cC) + 3} W_\cC).\]
Hence $\rad(W_\cC,\tau^{-\cW(\cC) + 3} W_\cC) \neq 0$, since \cite[4.10]{Wor1} together with $\cW(\cC) = 3$ implies that $W_\cC$ is not a brick. By \cite[1.7]{Ker1} it follows that $\rk(\cC) \leq - \cW(\cC) + 3$. The second inequality follows from the proof of \cite[3.1.3]{Wor1}.
\end{proof}
 
In \cite{Wor1}, the inequality $-\cW(\cC) \leq \rk(\cC)$ in conjunction with $\underline{t}(\cK_r) = - \infty$ was used to prove that $\sup \{ \cW(\cD) \mid \cD \in \cR \} = \infty$. We choose a different approach and study the number $\cW(\cC)$ to draw conclusions for $\rk(\cC)$. 

\begin{Remark}
There are regular components with $\rk(\cC_{i}) = 1$ and $\cW(\cC_{i}) = i$ for $0 \leq i  \leq 2$ \cite[3.3.1]{Wor1} and a component $\cD$ with $\rk(\cD) = 0$ and $\cW(\cD) =0$ \cite[3.3.3]{Bi1}. 
\end{Remark}

\section{Covering Theory}

\subsection{General Theory}

We follow \cite{Ri6} and \cite{Ri7} and consider the universal cover $C_r$ of the quiver $\Gamma_r$. The underlying graph of $C_r$ is an $r$-regular tree and $C_r$ has bipartite orientation. That means each vertex $x \in (C_r)_0$ is a sink or a source and $|n(x)| = r$. In the following we recall the construction of $C_r$.\\
For a quiver $Q = (Q_0,Q_1,s,t)$ with arrow set $Q_1$ we write $(Q_1)^{-1}:=\{ \alpha^{-1} \mid \alpha \in Q_1\}$ for the formal inverses of  $Q_1$. Moreover we extend the functions $s$ and $t$ to $(Q_1)^{-1}$ by defining
$s(\alpha^{-1}) := t(\alpha)$ and $t(\alpha^{-1}) := s(\alpha)$.
A walk $w$ in $Q_1$ is a formal sequence $w = \alpha^{\varepsilon_n}_n \cdots \alpha^{\varepsilon_1}_1$ with $\alpha_i \in Q_1$, $\varepsilon \in \{1,-1\}$ such that $s(\alpha^{\varepsilon_{i+1}}_{i+1}) = t(\alpha^{\varepsilon_{i}}_{i})$ for all $i < n$, where $\alpha^{1}:=\alpha$ for all $\alpha \in Q_1$. We set $t(w) := t(\alpha^{\varepsilon_n}_{n})$ and $s(w) := s(\alpha^{\varepsilon_1}_1)$.

\noindent Let $\sim$ be the  equivalence relation on the set of walks $W$  of $\Gamma_r$ generated by
\[\gamma^{-1}_i \gamma_i \sim \epsilon_1 \ \text{and} \ \gamma_i \gamma^{-1}_i \sim \epsilon_2.  \]
Let $^{-1} \colon W \to W$ be the involution on $W$ given by $(\alpha^{\varepsilon_n}_n \cdots \alpha^{\varepsilon_1}_1)^{-1} :=\alpha^{-\varepsilon_1}_1 \cdots \alpha^{-\varepsilon_n}_n$.
Now consider the fundamental group $\pi(\Gamma_r)$ of $\Gamma_r$ in the point $1$, i.e. the elements of $\pi(\Gamma_r)$ are the equivalence classes of unoriented paths starting and ending in $1$,  with multiplication given by concatenation of paths, inverse elements $[w]^{-1} := [w^{-1}]$ and identity element $[\epsilon_1]$. Note that $\pi(\Gamma_r)$ is a free group in the $r-1$ generators $\{[\gamma^{-1}_j \gamma_1] \mid 2 \leq j \leq r\}$ and in particular torsionfree.\\
The quiver $C_r$ is given by the following data:
\begin{enumerate}[topsep=0em, itemsep= -0em, parsep = 0 em, label=$(\alph*)$]
\item $(C_r)_0$ is the set of equivalence classes of paths starting in $1$.
\item There is an arrow from $[w]$ to $[w^\prime]$ whenever $w^\prime \sim \gamma_i w$ for some $i \in \{1,\ldots,r\}$.
\end{enumerate}

\begin{figure}[!h]
\centering 

\tikzstyle{level 1}=[sibling angle=90,  level distance=2.6cm]
\tikzstyle{level 2}=[sibling angle=90, level distance=1.2cm]
\tikzstyle{level 3}=[sibling angle=90, level distance=0.9cm]

\tikzstyle{every node}=[]
\tikzstyle{edge from child}=[]
\begin{tikzpicture}[grow cyclic,shape=circle,very thick,level distance=11mm,
                    cap=round]

                    \node {$\bullet$} child [color=black] foreach \A in {$\bullet$,$\bullet$,$\bullet$,$\bullet$}
    { node {\contour{white} {\A}} edge from parent[->] child [color=black] foreach \B in {$\bullet$,$\bullet$,$\bullet$}
        {  node {\contour{white} {\B}} edge from parent[<-] child [color=black] foreach \C in {$\bullet$,$\bullet$,$\bullet$}
            { node {\contour{white} {\C}} edge from parent[->]} 
        }
    };
    %\begin{scope}[on background layer]
    %\draw[black,dotted] (0,0) circle (43pt);
    %\draw[black,dotted] (0,0) circle (80pt);
    %\draw[black,dotted] (0,0) circle (120pt);
    %\end{scope}
\end{tikzpicture}
\caption{Illustration of $C_r$ for $r = 4$.}
\end{figure}
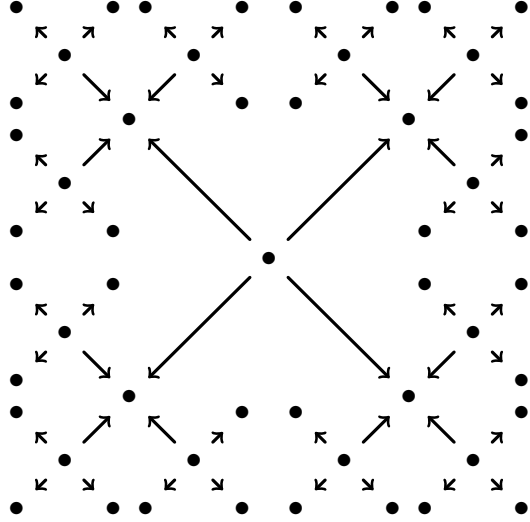

Let $\pi \colon C_r \to \Gamma_r$ be the quiver morphism  given by $[w] \mapsto t(w)$ and $([w] \to [\gamma_i w]) \mapsto \gamma_i$. The morphisms $\pi$ is a $G$-Galois cover for $G = \pi(\Gamma_r)$, where the action of $G$ on $C_r$ is given by concatenation of paths:
if $g = [w] \in \pi(\Gamma_r)$ and $[v],[u] \in (C_r)_0$ with arrow $[u] \to [\gamma_i u]$ then
\[ g.[v] = [v w^{-1}] \ \text{and}\]
\[ g.([u] \to [\gamma_i u]) = ([u w^{-1}] \to [\gamma_i u w^{-1}]).\]
We define $C_r^+ := \pi^{-1}(\{1\})$, $C_r^- := \pi^{-1}(\{2\})$ and get an induced action on $\rep(C_r)$ by shifting the support of representations via $G = \pi(\Gamma_r)$: Given $M \in \rep(C_r)$ and $g \in G$ we define $M^g:= (((M^g)_x)_{x \in (C_r)_0},(M^g(\alpha))_{\alpha \in (C_r)_1})$, where
\[ (M^g)_x := M_{g.x} \ \text{and} \ M^g(\alpha) := M(g.\alpha).\]

\noindent By identifying the orbit quiver $C_r/G$ with $\Gamma_r$ we define the push-down functor $\pi_{\lambda} \colon \rep(C_r) \to \rep(\Gamma_r)$ on the objects via
$\pi_{\lambda}(M) := (\pi_{\lambda}(M)_1,\pi_{\lambda}(M)_2; (\pi_{\lambda}(M)(\gamma_i))_{1 \leq i \leq r})$, where
\[ \pi_{\lambda}(M)_i := \bigoplus_{\pi(y) = i} M_y \ \text{and} \]
\[ \pi_{\lambda}(M)(\gamma_i) := \bigoplus_{\pi(\beta) = \gamma_i} M(\beta) \colon \pi_{\lambda}(M)_1 \to \pi_{\lambda}(M)_2 \ \text{for} \ 1 \leq i \leq r. \]
If $f = (f_x)_{x \in (C_r)_0} \colon M \to N$ is a morphism in  $\rep(C_r)$ then $\pi_{\lambda}(f) = (g_{\pi(x)})_{x \in (C_r)_0} = (g_1,g_2)$ with \[g_{i} := \bigoplus_{\pi(y) = i} f_y \colon \pi_{\lambda}(M)_{i} \to \pi_{\lambda}(N)_{i}.\]
By \cite[3.2]{Gab2} $\pi_{\lambda}$ is an exact functor.

\begin{Theorem}\cite[3.6]{Gab3}, \cite[6.2,6.3]{Ri7}\label{TheoremRingelGabriel}
The following statements hold.
\begin{enumerate}[topsep=0em, itemsep= -0em, parsep = 0 em, label=$(\alph*)$]
\item $\pi_{\lambda}$ sends indecomposable representations in $\rep(C_r)$ to indecomposable representations in $\rep(\Gamma_r)$.
\item If $M \in \rep(C_r)$ is indecomposable then $\pi_{\lambda}(M) \cong \pi_{\lambda}(N)$ if and only if $M^g \cong N$ for some $g \in G$.
\item $\pi_{\lambda}$ sends $AR$ sequences to $AR$ sequences and $\pi_{\lambda}$ commutes with the Auslander-Reiten translates, i.e. $\tau \circ \pi_{\lambda} = \pi_{\lambda} \circ \tau_{C_r}$.
\item If $M \in \rep(C_r)$ is indecomposable in a component $\cD$ with $\pi_{\lambda}(M)$ in a component $\cC$, then $\pi_{\lambda}$ induces a covering $\cD \to \cC$ of translation quivers. 
\end{enumerate}
\end{Theorem}

\begin{Definition}
Let $M \in \rep(C_r)$ be indecomposable. $M$ is called regular if $\pi_{\lambda}(M)$ is regular. A component $\cD$ of $\rep(C_r)$ is called regular if it contains a regular representation $M$. In this case we denote by $\pi_{\lambda}(\cD)$ the component containing $\pi_{\lambda}(M)$. Moreover we let $\cR({C_r})$ be the set of all regular components of the Auslander-Reiten quiver of $\rep(C_r)$.
\end{Definition}

\begin{corollary}
Let $\cD$ be regular component, then the covering $\cD \to \pi_{\lambda}(\cD)$ is an isomorphism of translation quivers. In particular, $\cD$ is of type $\ZZ A_\infty$. Moreover a component $\cE$ is regular if and only if $\cE$ is of type $\ZZ A_\infty$.
\end{corollary}
\begin{proof}
By \cite[1.7]{Gab2} $\pi_{\lambda}(\cD) \cong \ZZ A_\infty$ is a simply connected translation quiver. By \cite[1.6]{Gab2},\cite[1.7]{Ried1} the quiver morphism $\cD \to \pi_\lambda(\cD)$ is an isomorphism. If $\pi_{\lambda}(\cE) \in \{\cI,\cP\}$ then there exists a vertex $x$ with $r \geq 3$ successors $($see Figure $\ref{Fig:AuslanderReitenquiver})$, since a covering is surjective on arrows. Hence $\cE$ is not of type $\ZZ A_\infty$.
\end{proof}

\subsection{Duality}\label{Duality}
Recall \cite[2.2]{Wor1} that the duality $D \colon \rep(\Gamma_r) \to \rep(\Gamma_r)$ is defined by setting $(DM)_{x} := (M_{\psi(x)})^\ast$ and $(DM)(\gamma_i) := (M(\gamma_i))^\ast$, where $\psi \colon \{1,2\} \to \{1,2\}$  is the involution with $\psi(1) = 2$. \\
We define an involution $\varphi_0 \colon (C_r)_0 \to (C_r)_0$ via $[w] \mapsto [ \overline{w} \gamma_1]$, where $\overline{\epsilon_1} := \epsilon_2$, $\overline{\epsilon_2} := \epsilon_1$
and $\overline{\alpha^{\varepsilon_n}_n \cdots \alpha^{\varepsilon_1}_1} := \alpha^{-\varepsilon_n}_n \cdots \alpha^{-\varepsilon_1}_1$. This induces a quiver anti-morphism $\varphi \colon C_r \to C_r$ in the following way. \\
If $[w] \to [\gamma_i w]$ is an arrow of $C_r$, then by definition there is a unique arrow $\varphi([w] \to [\gamma_i w])$ starting in $\varphi_0([\gamma_i w]) = [\gamma^{-1}_i \overline{w} \gamma_1]$ and ending in $\varphi_0([w]) = [\overline{w} \gamma_1]$, since $\gamma_i \gamma^{-1}_i \overline{w} \gamma_1 \sim \overline{w}  \gamma_1$. Note that $\varphi(C_r^+)$ = $C_r^-$, $\varphi(C_r^-) = C_r^+$ and $\pi(\varphi(\alpha)) = \pi(\alpha)$.\\
We define a duality $D_{C_r} \colon \rep(C_r) \to \rep(C_r)$ by setting
$D_{C_r}M := ((D_{C_r}M)_{x \in (C_r)_0},(D_{C_r}M(\alpha))_{\alpha \in Q_1})$ where $(D_{C_r}M)_x := (M_{\varphi(x)})^\ast$ and $D_{C_r}M(\alpha) := (M(\varphi(\alpha)))^\ast$. By construction we have $\pi_{\lambda} \circ D_{C_r} = D \circ \pi_{\lambda}$.

\section{Lifting EKP and EIP to $\rep(C_r)$}
In the following we give a characterization of the equal images and equal kernels property for indecomposable representations of the form $\pi_{\lambda}(M)$. Let $ \bar{} \  \colon (C_r)_1 \to \{1,\ldots,r\}$ be the unique map with $\gamma_{\overline{\beta}} = \pi(\beta)$ for all $\beta \in (C_r)_1$. Note that if $x \in (C_r)^+ ($or $x \in C_r^-)$ then the restriction of $\ \bar{} \ $ to $\{ \alpha \in (C_r)_1 \mid s(\alpha) = x \}$ $($resp. $\{ \alpha \in (C_r)_1 \mid t(\alpha) = x\})$ is a bijective map to $\{1,\ldots,r\}$.

\begin{Definition}  Let $\emptyset \neq X \subseteq (C_r)_0$ be a set of vertices and $T \subseteq C_r$ be a tree.
\begin{enumerate}[topsep=0em, itemsep= -0em, parsep = 0 em, label=$(\alph*)$]
\item The unique minimal tree containing $X$ is denoted by $T(X)$.
\item A vertex $x \in T_0$ is called a leaf of $T$, if $|n(x) \cap T_0| \leq 1$. 
\end{enumerate} 
\end{Definition}
 
\begin{Definition} Let $x \in (C_r)_0$ and $M$ be a representation of $C_r$.
\begin{enumerate}[topsep=0em, itemsep= -0em, parsep = 0 em, label=$(\alph*)$]
\item The set $\supp(M) := \{ y \in (C_r)_0 \mid M_y \neq 0 \}$ is called the support of $M$.
\item For $V \subseteq \supp(M)$ we let $M_{V}$ be the induced representation with $\supp(M_{V}) = V$.
\item The vertex $x$ is a leaf of $M$ if $x$ is a leaf of $T(M) := T(\supp(M))$.
\item $M$ is called balanced provided that $M$ is indecomposable and $M$ has leaves in $C_r^+$ and $C_r^-$.
\end{enumerate} 
\end{Definition} 

\noindent Observe that if $M$ is indecomposable we have $T(M)_0 = \supp(M)$.
 
\begin{Definition}
We define  
\[ \Inj := \{ M \in \rep(C_r) \mid \forall \delta \in (C_r)_1: M(\delta) \ \text{is injective}\} \ \text{and} \]
\[\Sur := \{ M \in \rep(C_r) \mid \forall \delta \in (C_r)_1: M(\delta) \ \text{is surjective}\}.\]
\end{Definition} 
 
\begin{TheoremS}
\label{PropositionGeneral}
Let $M \in \rep(C_r)$ be an indecomposable representation. The following statements are equivalent:
\begin{enumerate}[topsep=0em, itemsep= -0em, parsep = 0 em, label=$(\alph*)$]
\item $N := \pi_{\lambda}(M)$ has the equal kernels property.
\item $N(\gamma_i)$ is injective for all $i \in \{1,\ldots,r\}$.
\item $M \in \Inj$.
\end{enumerate}
\end{TheoremS} 
\begin{proof}
$(a) \Rightarrow (b)$: Clear from the definition of $\EKP$.\\
$(b) \Rightarrow (c)$: Let $\alpha \colon x \to y$ and $m_x \in \ker M(\alpha)$. Denote by $\iota_x \colon M_x \to \bigoplus_{z \in C_r^+} M_z$ the natural embedding and let $\iota_y \colon M_y \to \bigoplus_{z \in C_r^-} M_z$. We conclude
\[0 = \iota_{y} \circ M(\alpha)(m_x) = \left[\bigoplus_{\pi(\beta) =\pi(\alpha)} M(\beta)\right] \circ \iota_x(m_x) = (N(\gamma_{\overline{\alpha}}) \circ \iota_x)(m_x).\] 

Hence $\iota_x(m_x) \in \ker N(\gamma_{\overline{\alpha}}) = \{0\}$ and $m_x = 0$.\\

\noindent $(c) \Rightarrow (a)$: Let $\alpha \in k^r \setminus \{0\}$, $f :=  \sum^r_{i=1} \alpha_i N(\gamma_i) \colon N_1 \to N_2$ and $m \in \ker f$. We assume w.l.o.g that $\alpha_1 \neq 0$. Write $m = (m_z)_{z \in C_r^+}$. We have to show that $m_z = 0$ for all $z \in C_r^+$. \\ 
Let $S := \{ z \in (C_r)_0 \mid m_z \neq 0\} \subseteq C_r^+$ and suppose that $S \neq \emptyset$. Let $T(S)$ be the minimal tree that contains $S$. Since $T(S)_0 \subseteq \supp(M)$ every leaf belongs to $C^+_r \cap S$, by the minimality of $T(S)$. Fix a leaf $x$ in $T(S)$ and let $\gamma \colon x \to y$ be the unique arrow with $\pi(\gamma) = \gamma_1$. Then 
 \[ 0 = (f(m))_y = \sum_{\beta \in (C_r)_1, t(\beta) = y} \alpha_{\overline{\beta}} M(\beta)(m_{s(\beta)}) = \alpha_1 M(\gamma)(m_x) + \sum_{\gamma \neq \beta \in (C_r)_1 , t(\beta) = y } \alpha_{\overline{\beta}} M(\beta)(m_{s(\beta)}).\]
By the injectivity of $M(\gamma)$ there is $\delta \colon z \to y \in (C_r)_1 \setminus \{\gamma\}$ with $t(\delta) = y$ and $0 \neq \alpha_{\overline{\delta}} M(\delta)(m_{z})$. It follows $m_z \neq 0$ and $z,x \in S$. Since $C_r$ is a tree we get $y \in T(S)_0$.\\
Since $\overline{\delta} \neq \overline{\gamma} = 1$, we assume without loss of generality that $\overline{\delta} = 2$, so that $\alpha_2 \neq 0$.  
Let $\eta \colon x \to a$ be the unique arrow with $\overline{\eta} = 2$. Then 
 \[ 0 = (f(m))_a = \sum_{\beta \in (C_r)_1, t(\beta) = a} \alpha_{\overline{\beta}} M(\beta)(m_{s(\beta)}) = \alpha_2 M(\eta)(m_x) + \sum_{\eta \neq \beta \in (C_r)_1 , t(\beta) = a } \alpha_{\overline{\beta}} M(\beta)(m_{s(\beta)}).\]
Hence there is $\zeta \colon b \to a \in (C_r)_1 \setminus \{\eta\}$ with $0 \neq \alpha_{\overline{\zeta}} M(\zeta)(m_b)$ and $a,b$ are in $T(S)_0$. We have shown that $a$ and $y$ are in $T(S)_0$. This is a contradiction since $x$ is a leaf. 
\end{proof}
 
\begin{corollaryS}\label{CorollaryCharEIP}
Let $M \in \rep(C_r)$ be an indecomposable representation. The following statements are equivalent.
\begin{enumerate}[topsep=0em, itemsep= -0em, parsep = 0 em, label=$(\alph*)$]
\item $N := \pi_{\lambda}(M)$ has the equal images property.
\item $N(\gamma_i)$ is surjective for all $i \in \{1,\ldots,r\}$.
\item $M \in \Sur$.
\end{enumerate}
\end{corollaryS}
\begin{proof}
$(c) \Rightarrow (a)$: Let $M(\alpha)$ be surjective for each $\alpha \in (C_r)_1$, then $(D_{C_r}M)(\alpha)$ is injective for each $\alpha \in (C_r)_1$. By $\ref{PropositionGeneral}$, the representation $ \pi_{\lambda}(D_{C_r} M) \cong D \pi_{\lambda} (M)$ has the equal kernels property. Therefore $\pi_{\lambda} (M)$ has the equal images property, since $D(\EKP) = \EIP$.
\end{proof}

\begin{corollaryS}
Let $\cD \in \cR(C_r)$ and $\cC := \pi_{\lambda}(\cD)$. Then there exist uniquely determined quasi-simple representations $I_\cD$ and $S_\cD$ in $\cD$ such that 
\begin{enumerate}[topsep=0em, itemsep= -0em, parsep = 0 em, label=$(\alph*)$]
\item  $\pi_{\lambda}(I_\cD) = M_\cC$ and $\pi_{\lambda}(S_\cD) = W_\cC$.
\item $\Sur \cap \cD = (\to S_\cD)$ and $\Inj \cap \cD = (I_\cD \to)$.
\item The unique integer $\cW_C(\cD)$ with $\tau_{C_r}^{\cW_C(\cD)+1}(I_{\cD}) = S_{\cD}$ is given by $\cW_C(\cD) = \cW(\cC) \in \NN_0$.
\end{enumerate}
\end{corollaryS}

\begin{corollaryS}
Assume that $M \in \rep(C_r)$ is balanced, then $M$ is regular.
\end{corollaryS}
\begin{proof}
If an indecomposable representation $X \in \rep(C_r)$ is in $\Inj$ $($respectively $\Sur)$, then all leaves of $X$ are in $C_r^+$ $($resp. $C_r^-)$. Hence $\pi_{\lambda}(M) \notin \EIP \cup \EKP$ by $\ref{PropositionGeneral}$. By \cite[2.7]{Wor1} the representations of the components $\cP$ and $\cI$ are contained in $\EKP \cup \EIP$.\end{proof}

\noindent From now on we write $x_0 := [\epsilon_1] \in (C_r)_0$ for the vertex in $(C_r)_0$ given by the trivial walk starting in the vertex $1$.

\begin{Definition}
For $i \in \{1,\ldots,r\}$, we let $\beta_i \in (C_r)_1$ be the unique arrow with $s(\beta_i) = x_0$ and $\pi(\beta_i) = \gamma_i$. Moreover let $z_i := t(\beta_i)$. We define an indecomposable representation $X^i$ in $\rep(C_r)$ via:
\[ (X^i)_y :=  
\begin{cases}
k, \ y \in \{x_0\} \cup {x^+_0} \setminus \{z_i\} \\ 
0, \ \text{else}
\end{cases}
\] 
and $X^i(\beta_j) := \id_k$ for all $j \neq i$. By definition we have $\dimu \pi_{\lambda} (X^i) = (1,r-1)$ and $\pi_{\lambda}(X^i) \cong X_{e_i}$.
\end{Definition}

\noindent In view of \cite[2.5]{Wor1}, \cite[3.6(c)]{Gab3} $\ref{TheoremRingelGabriel}$ and $\ref{PropositionGeneral}$ we conclude the following. 
\begin{corollaryS}\label{CorollaryEIPEKP}
Let $M \in \rep(C_r)$ be an indecomposable representation. The following statements are equivalent:
\begin{enumerate}[topsep=0em, itemsep= -0em, parsep = 0 em, label=$(\alph*)$]
\item $N := \pi_{\lambda}(M)$ has the equal kernels property.
\item $\Hom(\pi_{\lambda}(X^i),N) = 0$ for all $i \in \{1,\ldots,r\}$.
\item $\Hom({(X^i)}^g,M) = 0$ for all $i \in \{1,\ldots,r\}$ and all $g \in G$.
\end{enumerate}
\end{corollaryS}

\noindent The following Lemma will be needed later on.

\begin{LemmaS}\label{LemmaInjective} Let $n \in \NN$, $M$ be regular indecomposable, $i \in \{1,\ldots,r\}$ and $g \in G = \pi(\Gamma_r)$.
\begin{enumerate}[topsep=0em, itemsep= -0em, parsep = 0 em, label=$(\alph*)$]  
\item For $n \geq 2$, the linear map $(\tau^n_{C_r} X^i)^g(\alpha)$ is surjective for all $\alpha \in (C_r)_1$. 
\item If $n \geq 1$ and $f = (f_x)_{x \in (C_r)_0} \colon (\tau^{n}_{C_r} X^i)^g \to M$ is a non-zero morphism, then each $f_x$ is injective.
\item If $n \geq 2$, $x \in \supp((\tau^{n}_{C_r} X^i)^g) \cap C_r^-$ and $0 \neq f \colon (\tau^{n}_{C_r} X^i)^g \to M$, then $\supp(M) \cap n(x) = \supp(M) \cap x^- = x^-$. This means $|\supp(M) \cap x^-| = r$.
\end{enumerate}
\end{LemmaS}
\begin{proof}
$(a)$ For $n \geq 2$, we have $\pi_{\lambda}((\tau^n_{C_r} X^i)^g) = \pi_{\lambda}(\tau^n_{C_r} X^i) =  \tau^n X_{e_i} \in \EIP$ $($see \cite[3.3]{Wor1}$)$.\\
$(b)$ For $n \geq 1$, it is known that each proper factor of $\tau^{n} X_{e_i} $ is preinjective, see for example \cite[2.1.4]{Bi1}. Let $0 \neq f \in \Hom((\tau_{C_r}^{n} X^i)^g,M)$, then $0 \neq \pi_{\lambda}(f) = (g_i)_{1 \leq i \leq 2} \colon \tau^{n} X_{e_1} \to \pi_{\lambda}(M)$ is injective, where
\[g_i = \bigoplus_{\pi(y) = i} f_y  \colon \bigoplus_{\pi(y) = i} ((\tau_{C_r}^n X^i)^g)_y \to \bigoplus_{\pi(y) = i} M_y.\]
So each $f_x$ is injective, since $\pi^{-1}(\{1,2\}) = (C_r)_0$.\\
$(c)$ Let $x \in \supp((\tau^{n}_{C_r} X^i)^g) \cap C_r^-$ and  $z \in n(x)$. Since $x$ is a sink, there is $\alpha \colon z \to x$ and by $(a)$, $(\tau^n_{C_r} X^i)(\alpha)$ is surjective. Hence $((\tau^n_{C_r} X^i)^g)_z \neq 0$. By $(b)$, we get $M_z \neq 0$ and $z \in \supp(M) \cap x^-$.
\end{proof}

\section{Considerations in the universal covering}

Let us recall what we have shown so far. Given a component $\cD$ in $\cR(C_r)$, the natural number $\cW_C(\cD)$ is the distance between the two non-empty, non-intersecting cones $\Inj \cap \cD$ and $\Sur \cap \cD$. Moreover we know that $\cW(\pi_{\lambda}(\cD)) = \cW_C(\cD)$. \\
Let now $X \in \cD$ be indecomposable. Then there exists an integer $l \in \ZZ$ such that $\tau^l_{C_r} X \in \Sur$, since $\Sur \cap \cD$ is non-empty. We also find $n \geq 1$ with $\tau^{-n}_{C_r} X \not\in \Sur$. 
Since $\Sur \cap \cD$ is closed under $\tau_{C_r}$ we conclude $-n \leq l$. Therefore the following minima exist.

\begin{Definition}
Let $X \in \rep(C_r)$ be a regular indecomposable representation. We define 
\[ d^{-}(X) := \min \{ l \in \ZZ \mid \tau^{-l} X \in \Inj \} \ \text{and} \ d^{+}(X) := \min \{ l \in \ZZ \mid \tau^{l} X \in \Sur\}. \]
\end{Definition}

\noindent Note that $|d^-(X)| \in \NN_0$ is the distance of $X$ to the border of the cone $\Inj \cap \cD$. 

\begin{LemmaS}
Let $X \in \rep(C_r)$ be indecomposable in a regular component $\cD$. Then \[\cW_C(\cD) = d^+(X) + d^-(X) - \ql(X).\] 
\end{LemmaS}
\begin{proof}
Since the equality is obvious for $X$ quasi-simple we assume $l:= \ql(X) > 1$. Let $Z$ be the unique quasi-simple representation with $X = Z[l]$. By  induction we get $\cW_C(\cD) = d^+(Z[l-1]) + d^-(Z[l-1]) - (l-1)$. Now observe that $d^+(Z[l-1]) +1 = d^+(Z[l])$ and $d^-(Z[l-1]) = d^-(Z[l])$. Hence $\cW_C(\cD) = d^+(Z[l]) - 1 + d^-(Z[l]) - (l-1) = d^+(X) + d^-(X) - \ql(X)$.
\end{proof}

\begin{figure*}[!h]
\centering 
\includegraphics[width=0.8\textwidth, height=125px]{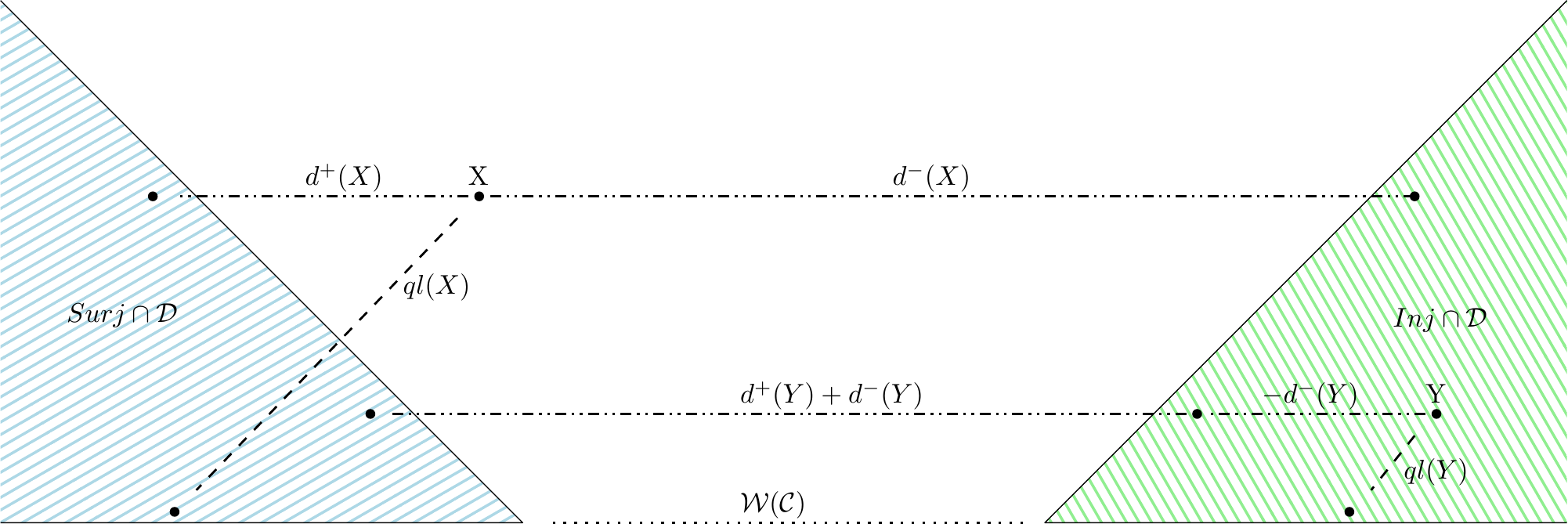}
\caption{How to determine $\cW_C(\cD)$ using an arbitrary representation in $\cC$.}
\label{Fig:GeneralWidth}
\end{figure*}

 We show in the following how to modify a regular representation $X \rightsquigarrow X^\prime$ such that the obtained representation $X^\prime$ is regular, $d^+(X) = d^+(X^\prime)$, $d^-(X) = d^-(X^\prime)$ and $\ql(X^\prime) = 1$.

\subsection{Indecomposable representations arising from extensions}

In this section, we show how to construct non-split exact sequences with indecomposable middle term in $\rep(C_r)$.
 
\begin{Definition}

Let $M, N \in \rep(C_r)$ be indecomposable. The pair
$(N,M)$ is called leaf-connected if there is $\alpha \colon x \to y \in (C_r)_1$ s.t. 
\begin{enumerate}[topsep=0em, itemsep= -0em, parsep = 0 em, label=$(\alph*)$]
\item $x$ is a leaf of $M$, $y$ is a leaf of $N$ and
\item $\supp(M) \cap \supp(N) = \emptyset$.
\end{enumerate}
\end{Definition}

\begin{Remark} Note that the assumption $M$ and $N$ being indecomposable together with properties $(a)$ and $(b)$ already implies the uniqueness of $\alpha$. If $(N,M)$ is leaf-connected, $\alpha$ is called the connecting arrow and $(M,N)$ is not leaf-connected.
\end{Remark}

\begin{Definition}
Let $(N,M)$ be leaf-connected with connecting arrow $\alpha \colon x \to y$ and $f \colon M_x \to N_y$ a non-zero linear map. We define a representation $N \ast_f M \in \rep(C_r)$ by setting
\[\supp(N\ast_f M) := \supp(N) \cup \supp(M) \]
with $(N\ast_f M)_{\supp (X)} = X$ for $X \in \{M,N\}$ and $ (N\ast_f M)(\alpha) := f \colon M_x \to N_y$. 
Moreover we denote by $\iota_f \colon N \to N \ast_f M$ and $\pi_{f} \colon N \ast_f M \to M$ the natural morphisms of quiver representations.  The $k$-linear map along the connecting arrow is called a connecting map for $(N,M)$.
\end{Definition}

\begin{Remark}
Note that $N \ast_f M$ is just an extension of $M$ by $N$ with corresponding exact sequence
$\delta_f \colon 0 \to N \stackrel{\iota_f}{\to} N \ast_f M \stackrel{\pi_f}{\to} M \to 0$. The next lemma shows that $M \ast_f N$ is indecomposable. In particular, $\delta_f$ does not split. One can also show that the map 
\[ \Phi \colon \Hom_k(M_x,N_y) \to \Ext(M,N), f \mapsto [\delta_f], \]
is an isomorphism of vector spaces, where $[\delta_0]$ is the neutral element in the abelian group $\Ext(M,N)$ with respect to the Baer sum. 
\end{Remark}

\begin{Lemma}\label{LemmaProperties} Let $(N,M)$ be leaf-connected.
\begin{enumerate}[topsep=0em, itemsep= -0em, parsep = 0 em, label=$(\alph*)$]
\item The representation $N \ast_f M$ is indecomposable.
\item If $M$ and $N$ are regular, then $N \ast_f M$ is regular.
\end{enumerate}
\end{Lemma}
\begin{proof}
$(a)$ Let $U_1,U_2 \in \rep(C_r)$ with $N \ast_f M = U_1 \oplus  U_2$. Hence we get
\[ M = (N \ast_f M)_{\supp(M)} = {(U_1)}_{\supp(M)} \oplus {(U_2)}_{\supp(M)}. \]
Since $M \in \rep(C_r)$ is indecomposable, there is a unique $i \in \{1,2\}$ with $(U_{i})_{\supp(M)} = M$. We assume $i = 1$. By the same token there is a unique $j$ with $(U_{j})_{\supp(N)} = N$.  Since $(M \ast_f N)(\alpha) = f \neq 0$, it follows that $j = i = 1$. Hence $U_2 = (0)$ and $N \ast_ f M$ is indecomposable.\\
$(b)$ By construction $N$ is a subrepresentation of $M \ast_f N$ and $M$ a factor representation. Since $\pi_{\lambda}(N \ast_f M)$ is indecomposable with regular factor representation $\pi_{\lambda}(M)$, it is regular itself or preprojective. By the same token $\pi_{\lambda}(N \ast_f M)$ is regular or preinjective.
\end{proof}

\begin{Lemma}\label{LemmaConnected}
Let $M,N \in \rep(C_r)$ be indecomposable  and assume $x \in C_r^+$ is a leaf of $M$ and $y \in C_r^-$ is a leaf of $N$. Then there exists $g \in G$ such that $(N^g,M)$ is leaf-connected with connecting arrow $x \to g.y$.
\end{Lemma}

\begin{proof} Since $x,y$ are leaves, we find at most one arrow $\delta_M \colon x \to x_1$ and at most one arrow $\delta_N \colon y_1 \to y$ with $x_1 \in \supp(M)$ and  $y_1 \in \supp(N)$. Since  $r \geq 3$ we find an arrow $\alpha \colon x \to z$ with $\pi(\delta_M) \neq \pi(\alpha) \neq \pi
(\delta_N)$. In particular $\alpha \neq \delta_M$ and $x_1 \neq z$.\\
Now let $g \in \pi(\Gamma_r)$ be the unique element with $g.z = y$. Then $z \in n(x)$ is a leaf of $N^g$. By construction, we have $\pi(g.\alpha) = \pi(\alpha) \neq \pi(\delta_N)$ and conclude $x \notin \supp(N^g)$. It follows that $\supp(M) \cap \supp(N^g) = \emptyset$. Hence $(N^g,M)$ is leaf-connected with connecting arrow $\alpha$. 
\end{proof}

\begin{Definition}
 Let $n \geq 2$ and $M_1,\ldots,M_n \in \rep(C_r)$ be indecomposable. The tuple $(M_n,\ldots,M_1)$ is called leaf-connected, provided that $(M_{i+1},M_{i})$ is leaf-connected for all $1 \leq i < n$. A tuple  $(f_{n-1},\ldots,f_1)$ of $k$-linear maps is called a connecting map for $(M_n,\ldots,M_1)$ if $f_i$ is a connecting map for $(M_{i+1},M_i)$ for all $1 \leq i < n$.
\end{Definition}

\begin{Remark}
Let $(M,L)$ and $(L,N)$ be leaf-connected with connecting maps $f$ and $g$, then $\supp(M) \cap \supp(N) = \emptyset$, since $C_r$ is a tree. Hence $\supp(M \ast_f L) \cap \supp(N) = \emptyset$ and $(M \ast_f L,N)$ is connected. Therefore the next definition is well-defined. 
\end{Remark}

\begin{Definition}
Let $n \geq 2$, $(M_n,\ldots,M_1)$ be leaf-connected with connecting map $(f_{n-1},\ldots,f_1)$. Then we define inductively $M_n \ast_{f_{n-1}} M_{n-1} \ast_{f_{n-2}} \cdots \ast_{f_i} M_{i} := (M_n \ast_{f_{n-1}} M_{n-1} \ast \cdots \ast_{f_{i+1}} M_{i+1}) \ast_{f_i} M_{i}$ for all $2 \leq i < n$. 

\end{Definition}

From now on, we assume that $(M_n,\ldots,M_1)$ is leaf-connected with connecting map $(f_{n-1},\ldots,f_1)$. For $1 \leq i \leq n$, we define $\ast_{j \geq i} M_j := M_n \ast_{f_{n-1}} M_{n-1} \ast \cdots \ast_{f_i} M_{i}$ and $\ast_{j \leq i} M_j := M_i \ast_{f_{i-1}} M_{i-1} \ast \cdots \ast_{f_1} M_{1}$. Moreover, we set $\ast_{j \geq n} M_j = M_n $ and $\ast_{j \leq 1} M_j = M_1$.

\begin{Lemma}\label{LemmaProperties1} Let $n \geq 2$, $(M_n,\ldots,M_1)$ be leaf-connected and $1 \leq i < n$.
\begin{enumerate}[topsep=0em, itemsep= -0em, parsep = 0 em, label=$(\alph*)$]
\item The representation $\ast_{j \geq i} M_j$ is indecomposable.
\item There is a short exact sequence $0 \to \ast_{j \geq {i+1}} M_j \to \ast_{j \geq i} M_j\to M_i \to 0$.
\item If $M_n,M_1$ are regular, then $\ast_{j \geq i} M_j$ is regular.
\end{enumerate}
\end{Lemma}
\begin{proof}
For $(c)$, just note that $M_i$ is balanced for $2 \leq i \leq n -1$, hence $M_i$ is regular.
\end{proof}

\begin{Lemma}\label{LemmaTau}
Let $X,Y \in \rep(C_r)$ be regular indecomposable. Then  the following statments are equivalent.
\begin{enumerate}[topsep=0em, itemsep= -0em, parsep = 0 em, label=$(\alph*)$]
\item There is $g \in G$ such that $\Hom(X^g,Y) \neq 0$.
\item There is $h \in G$ such that  $\Hom(\tau_{C_r} X^h,\tau_{C_r} Y)$.
\item There is $l \in G$ such that $\Hom(\tau_{C_r}^{-1} X^l,\tau_{C_r}^{-1} Y)$.
\end{enumerate}
\end{Lemma}
\begin{proof} We only show $(a) \Rightarrow (b)$.
Let $g \in G$, such that $0 \neq \Hom(X^g,Y)$. Then 
\[0 \neq \Hom(\pi_{\lambda}(X),\pi_{\lambda}(Y)) \cong \Hom(\tau \circ \pi_{\lambda}(X),\tau \circ \pi_{\lambda}(Y)) \cong \Hom(\pi_{\lambda}(\tau_{C_r} X),\pi_{\lambda}(\tau_{C_r} Y)).\]
 Hence we find $h \in G$ such that $0 \neq \Hom((\tau_{C_r} X)^h,\tau_{C_r} Y) \cong \Hom(\tau_{C_r} X^h,\tau_{C_r} Y)$.
\end{proof}

\begin{proposition}\label{CorollaryProperties} Let $n \geq 2$, $(M_n,\ldots,M_1)$ be leaf-connected and $M_1,M_n$ regular, then 
\begin{enumerate}[topsep=0em, itemsep= -0em, parsep = 0 em, label=$(\alph*)$]
 \item  
     $\max\{ d^-(\ast_{j \geq i} M_j) \mid 1 \leq i \leq n\} \leq d^-(\ast_{j \geq 1} M_j) \leq \max\{d^-(M_i) \mid 1 \leq i \leq n \}$.
\item $\max\{ d^+(\ast_{j \leq i} M_j) \mid 1 \leq i \leq n\} \leq d^+(\ast_{j \geq 1} M_j) \leq \max\{d^+(M_i) \mid 1 \leq i \leq n \}$.
\end{enumerate}     
\end{proposition}
\begin{proof}
Note that we have a filtration of $\ast_{j\geq 1}M_j$ by regular subrepresentations
\[ 0 \subset M_n \subset M_n \ast M_{n-1} \subset M_n \ast M_{n-1} \ast M_{n-2} \subset \cdots \subset \ast_{j \geq 2} M_j \subset \ast_{j \geq 1} M_j,\] 
with $\ast_{j \geq i} M_j / \ast_{j \geq i+1} M_j \cong M_i$ regular for all $1 \leq i \leq n$, where $\ast_{j \geq n+1} M_j := 0$. \\
$(a)$ Let $i \in \{1,\ldots,n\}$ and $Z := \ast_{j \geq i} M_i$. Consider the short exact sequence 
\[ 0 \to Z \to \ast_{j\geq 1} M_j \to  (\ast_{j\geq 1} M_j)/Z \to 0.\]
Now $l \in \ZZ$ such that $\tau_{C_r}^{-l} Z \notin \Inj$. Then there exists $i \in \{1,\ldots,r\}$ and $g \in G$  such that $\Hom((X^i)^g,\tau_{C_r}^{-l} Z) \neq 0$. Hence we find $h \in G$ with $\Hom((\tau_{C_r}^l X^i)^h,Z)$ for some $h \in G$. Left-exactness of $\Hom((\tau_{C_r}^l X^i)^h,-)$ ensures that $0 \neq \Hom((\tau_{C_r}^l X^i)^h,\ast_{j \geq 1}M_j)$ and therefore we find $g \in G$ with $0 \neq \Hom(X^l_i,\tau_{C_r}^{-l} (\ast_{j \geq 1} M_j))$. Hence $\tau_{C_r}^{-l} (\ast_{j \geq 1} M_j) \notin \Inj$ and therefore $d^-(X) \leq d^-(\ast_{j \geq 1} M_j)$.\\
If $\Hom((X^i)^g,\tau_{C_r}^{-l} (\ast_{j \geq 1} M_j)) \neq 0$ for some $g \in G$ and $i \in \{1,\ldots,r\}$, then we get find $h \in G$ with $0 \neq \Hom(\tau_{C_r}^l (X^i)^h,\ast_{j \geq 1} M_j)$.
By \cite[1.9]{Ker3}  we find $1 \leq p \leq n$ with $0 \neq \Hom(\tau_{C_r}^l (X^i)^h,M_p)$. Hence there is $g \in G$ with $0 \neq \Hom((X^i)^g,\tau_{C_r}^{-l} M_p)$. Hence $d^-(\ast_{j \geq 1} M_j) \leq d^-(M_p)$.  \\
%$(b)$ follows by $(a)$ in conjunction with the equality $D_{C_r} \circ d^+ = d^-$. 
$(b)$ Note for $i \geq 2$ that $\ast_{j \geq 1} M_j/ \ast_{j \geq i} M_j \cong \ast_{j < i} M_j$. Hence $D_{C_r}(\ast_{j\geq 1} M_j)$ has a filtration \[ 0 \subset D_{C_r}M_1 \subset D_{C_r}M_1 \ast D_{C_r}M_{2} \subset \cdots \subset \ast_{j \leq n-1} D_{C_r}M_j \subset D_{C_r}(\ast_{j\geq 1} M_j)\] 
Now apply $(a)$ and note that $d^+(X)=d^-(D_{C_r}X)$ for each regular indecomposable representation $X$.
\end{proof}

\subsection{Small representations and trees}

\begin{Definition}
A balanced representation $N$ is called small if $1 \leq d^-(N),d^+(N) \leq 2$. 
\end{Definition}

Note that $N$ being balanced always implies $d^-(N) \geq 1$ and $d^+(N) \geq 1$.

\label{DefinitionMetric}

\begin{Definition}
Denote with $\overline{C_r}$ the underlying graph of $C_r$. Then $((C_r)_0,d)$ obtains the structure of a metric space, where $d(x,y) \in \NN_0$ denotes the length of the unique path in $\overline{C_r}$ connecting vertices $x$ and $y$.
\end{Definition}

\begin{Definition}
 Let $T \subseteq C_r$ be a finite subtree. $T$ is called small if
  \begin{enumerate}[topsep=0em, itemsep= -0em, parsep = 0 em, label=$(\alph*)$]
    \item $T$ has leaves in $C_r^+$ and $C_r^-$,
    \item for all $x \in T_0$, we have $|T_0 \cap n(x)| \leq 3$,
    \item if $|T_0 \cap n(x)| = 3 = |T_0 \cap n(y)|$ then $x = y$ or $d(x,y) \geq 3$.
  \end{enumerate}
\end{Definition}

\begin{examples}

\noindent Let $l \in \NN$ and $n \in 2\NN$ with $n \geq 4l $. We denote with $\A_{l,n} \subseteq C_r$ a $($small$)$ subtree of the following form:
\begin{center}
\begin{tikzpicture}[very thick]
                    [every node/.style={fill, circle, inner sep = 1pt}]

\def \n {9}
\def \m {4}
\def \radius {3cm}
\def \margin {8} % margin in angles, depends on the radius

\foreach \s in {1,...,\n}
{
  \node[color=black] at ({\s},0) {$\bullet$};
  \node[color=black] at ({\s},-0.3) {$a_{\s}$}; %%% Label 1 ... n   
   }

\foreach \s in {1,...,\m}
{
    \draw[->, >=latex] (2*\s + 0.2,0) to (2*\s+ 1 - 0.2,0); %%%from left to right
    \draw[->, >=latex] (2*\s - 0.2,0) to (2*\s- 1 + 0.1,0); %%%from right to left
    %\node[color=black] at ({4*\s-1},1) {$\bullet$}; %%% 
    %\node[color=black] at ({4*\s-1},1+0.3) {$t_{\s}$}; %%%
    %\draw[->, >=latex] ({4*\s-1},1 - 0.2) to ({4*\s-1},0 + 0.2);
   }
   
%%%% Label %%%%%   
   
    \node[color=black] at ({4*1-1},1) {$\bullet$}; %%% 
    \node[color=black] at ({4*2-1},1) {$\bullet$}; %%% 
    \node[color=black] at ({4*3-0.2},1) {$\bullet$}; %%% 
    \node[color=black] at ({4*1-1},1+0.3) {$t_{1}$}; %%%
    \node[color=black] at ({4*2-1},1+0.3) {$t_{2}$}; %%% 
    \node[color=black] at ({4*3-0.2},1+0.3) {$t_l$}; %%%  
   
 	\node[color=black] at ({\n +1.8},-0.3) {$a_{4l-2}$};
 	\node[color=black] at ({\n +2.8},-0.3) {$a_{4l-1}$};
 	\node[color=black] at ({16.6},-0.3) {$a_{n}$};

    \draw[->, >=latex] ({4*1-1},1 - 0.2) to ({4*1-1},0 + 0.2);
    \draw[->, >=latex] ({4*2-1},1 - 0.2) to ({4*2-1},0 + 0.2);
    \draw[->, >=latex] ({11+0.8},1 - 0.2) to ({11+0.8},0 + 0.2);
 	\node[color=black] at ({\n+0.4},0) {...};
 	\node[color=black] at ({\n +0.8},0) {$\bullet$};
 	\node[color=black] at ({\n +1.8},0) {$\bullet$};
 	\node[color=black] at ({\n +2.8},0) {$\bullet$};
    \node[color=black] at ({\n +3.8},0) {$\bullet$};
 	 \draw[->, >=latex] (2*5 +0.6 ,0) to (2*5,0);
 	 \draw[->, >=latex] (2*5 +1 ,0) to (2*5+ 1.6,0);
 	 \draw[->, >=latex] (2*5 +2.6 ,0) to (2*5+2,0);
 	 \node[color=black] at ({13.2},0) {...};
 	 \node[color=black] at ({13.6},0) {$\bullet$};
     \node[color=black] at ({14.6},0) {$\bullet$};
     \node[color=black] at ({15.6},0) {$\bullet$};
     \node[color=black] at ({16.6},0) {$\bullet$};
     \draw[->, >=latex] (14.4,0) to (13.8,0);     
     \draw[->, >=latex] (14.8,0) to (15.4,0);
     \draw[->, >=latex] (16.4,0) to (15.8,0);
   
\end{tikzpicture}
\end{center}

\end{examples}

\begin{Lemma}
Assume $L$ is an indecomposable representation with small tree $T(L)$. Then $L$ is small.
\end{Lemma}
\begin{proof}
Since $T(L)$ is small with $T(L)_0 = \supp(L)$, $L$ is balanced and therefore regular.\\
Let $l \in \ZZ$ be such that $\tau^{-l}_{C_r} L \notin \Inj$. By $\ref{CorollaryEIPEKP}$ and $\ref{LemmaTau}$ we find $i \in \{1,\ldots,r\}$ and $g \in G$ with $0 \neq  \Hom((\tau_{C_r}^l X^i)^g,L)$. Fix $h \colon (\tau_{C_r}^l X^i)^g \to L$ non-zero. \\
We assume that $l \geq 2$. By $\ref{LemmaInjective}$, $h$ is a monomorphism  and $\supp((\tau_{C_r}^l X^i)^g) \subseteq \supp(L)$. Let $s$ be a sink of $(\tau_{C_r}^l X^i)^g$. By $\ref{LemmaInjective}$, we have $|\supp(L) \cap n(s)| = r$.
 Since $T(L)$ is small we get $3 \leq r = |\supp(L) \cap n(s)| \leq 3$. Hence 
 \[ (\ast) \quad r = 3 = |n(s) \cap \supp(L)|.  \] 
Now let $t_1,t_2 \in \supp((\tau_{C_r}^l X^i)^g)$ be sinks. Then $(\ast)$ yields 
\[|n(t_1) \cap \supp(L)| = 3 = |n(t_2) \cap \supp(L)|.\]
 Since $T(L)$ is small we get $t_1 = t_2$ or $d(t_1,t_2) \geq 3$. Since $C_r$ has bipartite orientation and $\supp((\tau_{C_r}^l X^i)^g)$ is connected, it follows $t_1 = t_2$. Hence $\supp((\tau_{C_r}^l X^i)^g)$ contains exactly one sink $s$. Write $n(s) = \{a,b,c\}$. Since $l\geq 2$,  $\ref{LemmaInjective}(a)$ implies $\supp((\tau_{C_r}^l X^i)^g) = \{s,a,b,c\}$.
Hence $Z := (\tau_{C_r}^l X^i)^g$ can be considered as a representation of the Dynkin diagram $D_4$ with unique sink $s$  such that all linear maps are surjective. It follows that $Z_x = 1$ for all $x \in \{s,a,b,c\}$. Hence $\dimu \pi_{\lambda}(Z) = (3,1)$ and $\pi_{\lambda}(Z)$ is indecomposable. But the only  indecomposable representation $I \in \rep(\Gamma_3)$ with dimension vector $(3,1)$ is injective. This is a contradiction since $Z$ is regular. 
 Therefore $l \leq 1$ and $d^-(L) \leq 2$. For the other inequality note that $T(L)$ is small if and only if $T(D_{C_r}L)$ is small and $d^+(L) = d^-(D_{C_r}L) \leq 2$.
\end{proof}

\section{The main theorem}

\begin{LemmaS}\label{LemmaBoChen}
Let $M \in \rep(\Gamma_r)$ be an indecomposable representation with dimension vector $\dimu M = (a+1,a)$, $a \geq 1$. Then $M$ is a regular and quasi-simple representation.
\end{LemmaS}
\begin{proof}
That $M$ is regular follows immediatly from \cite[2.1]{BoChen1}.
By \cite[3.4]{BoChen1} it suffices to show that $A_t$ is not a common divisor of $a+1$ and $a$ for all $t \geq 2$, where $\dimu P_i = (A_{i-1},A_i)$ is the dimension vector of the preprojective indecomposable representation $P_i$ $($see Figure $\ref{Fig:AuslanderReitenquiver})$. But this is trivial since $\ggt(a+1,a) = 1$.
\end{proof}

\begin{TheoremS}\label{TheoremMainTheorem}
Let $M \in \rep(C_r)$ be balanced in the regular component $\cD$ and $2 \leq d^+(M),d^-(M)$. There is $n_0 \in \NN$ such that for each $n \geq n_0$ there is a regular component $\cD_n$ with 
\[\cW_C(\cD_n) = \cW_C(\cD) + \ql(M) - 1.\] 
Moreover $\cD_n$ contains a balanced quasi-simple  representation $F_n$ with $\dimu \pi_{\lambda}(F_n) =(n+1,n)$ or $\dimu \pi_{\lambda}(F_n) = (n,n+1)$ and $D_i \neq D_j$ for $i \neq j \geq n_0$.
\end{TheoremS}

\begin{proof}
Write $\dimu \pi_{\lambda}(M) = (a,b)$. After dualising $M$ we can assume $a \leq b$. Set $l := 2(b-a) + 1 \geq 1$ and $p_0 := 4l$. Now let $p \geq p_0$ with $p \in 2\NN$. Consider an indecomposable and thin representation $L$ $($i.e. $\dim_k L_x \leq 1$ for all $x \in (C_r)_0$ such that $(L,M)$ is leaf-connected and $T(L) = T(\supp(L))$ is of type $\A_{l,p}$. Let $g \in G$ be such that $(M^g,L)$ is leaf-connected.  We conclude for $n:= 2b + \frac{1}{2}p \in \NN$, $n_0 :=2b+ \frac{1}{2}p_0$, and $F_n := M^g \ast L \ast M$ that
\begin{align*}
  \dimu \pi_{\lambda}(F_n) &= \dimu \pi_{\lambda}(M^g \ast L \ast M) = 2\dimu \pi_{\lambda}(M) + \dimu \pi_{\lambda}(L) =  2(a,b) + (\frac{1}{2}p+l,\frac{1}{2}p) \\
  &=(2a+2b-2a + 1,2b) + \frac{1}{2}(p,p) = (n+1,n).
\end{align*}
By $\ref{LemmaProperties1}$ $F_n$ is a regular indecomposable representation and by $\ref{LemmaBoChen}$ $\pi_{\lambda}(F_n)$ is quasi-simple. Therefore $F_n$ is quasi-simple in a regular component $\cD_n$. We conclude with $\ref{CorollaryProperties}$
\[ d^-(M) = d^-(M^g) \leq d^-(F_n) \leq \max\{d^-(M),d^-(L)\} = \max \{d^-(M),2\} = d^-(M),\]
i.e. $d^-(M) = d^-(F_n)$. By the same token we have $d^+(M) = d^+(F_n)$ and conclude 
\begin{align*}
  \cW_C(\cD_n) &= d^+(F_n) + d^-(F_n) - \ql(F_n) = d^+(M) + d^-(M) - 1 \\
            &= (d^+(M) + d^-(M) - \ql(M)) + \ql(M) -1 
           = \cW_C(\cD) + \ql(M) - 1.       
\end{align*}
It follows immediatly from the construction that the regular components are pairwise distinct, since $F_i$, $F_j$ are non-isomorphic and satisfy $\ql(F_i) = \ql(F_j)$ and $d^-(F_i) = d^-(F_j)$ for $i \neq j \geq n_0$.
\end{proof}
            
\begin{corollaryS}\label{CorollarySpecialmodule}
Let $M \in \cD$ be balanced and $2 \leq d^+(M),d^-(M)$.  Then there exists a balanced and quasi-simple representation $F$ in a regular component $\cE$  such that $\cW_C(\cE) = \cW_C(\cD) + \ql(M) - 1$. Moreover there is a leaf $x \in C_r^+$ with $\dim_k F_x = 1 = \dim_k F_y$ for the unique element $y \in x^+ \cap \supp(F)$.
\end{corollaryS}            
\begin{proof}
Fix $n \geq n_0$ in the proof of the theorem and set $F := F_n = M^g \ast L \ast M$. The last claim follows since $L$ is a thin representation of type $\mathbb{A}_{l,p}$ which has $l+1 \geq 2$ leaves in $C_r^+$ $($see Figure $\ref{Fig:ExampleSupport})$.
\end{proof}

\begin{figure}[!h]
\centering 
\includegraphics[width=0.9\textwidth, height=75px]{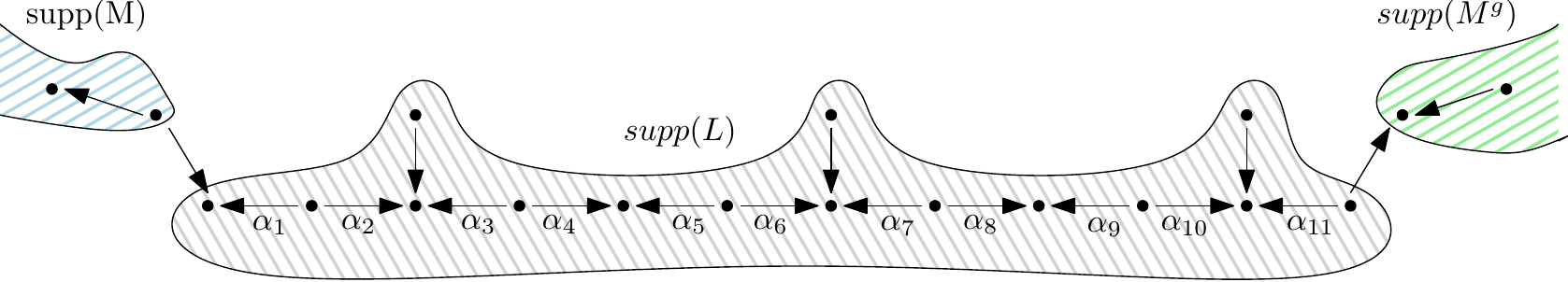}
\caption{Illustration of the proof with $\supp(L)$ of type $\A_{3,12}$}
\label{Fig:ExampleSupport}
\end{figure}

\section{Applications}

\subsection{Regular components for every width}

\noindent The aim of this section is to construct for each $n \in \NN$ a regular component $\cD$ with $\cW_C(\cD) = n$. 
Although each indecomposable representation has a leaf, it is in general not true that a regular representation has leaves in $C_r^+$ and $C_r^-$. For example if $M$ is indecomposable in $\Inj$, then each leaf of $M$ is a sink by $\ref{PropositionGeneral}$.
The next results shows that self-dual representations have leaves in $C_r^+$ and $C_r^-$.

\begin{Lemma}\label{LemmaDual}
Let $M \in \rep(C_r)$ be indecomposable such that $D\pi_{\lambda}(M) \cong \pi_{\lambda}(M)$, then $M$ is balanced.
\end{Lemma}
\begin{proof}
Since $\supp(M)$ is finite there exists a leaf $x$ of $M$. Without loss of generality we assume that $x \in C_r^+$. We get $\pi_{\lambda}(M) \cong D\pi_{\lambda}(M) \cong \pi_{\lambda}(D_{C_r}M)$. Therefore we find $h \in G$ such that $M \cong (D_{C_r}M)^h.$ Since $h^{-1}.\varphi(x) \in C_r^-$ $($see $\ref{Duality})$ is a leaf of $(D_{C_r}M)^h$ the claim follows.\end{proof}

We denote with $\sigma C_r$ the quiver obtained by changing the orientation of all arrows in $C_r$. Note that  $\sigma^2 C_r = C_r$ and $\sigma C_r \cong C_r$. We denote by $\Phi^+$ the composition of the Bernstein-Gelfand-Ponomarev reflection functors \cite[VII 5.5.]{Assem1} for all the sources of $C_r$. $\Phi^+$ is  a well-defined functor $\Phi^+ \colon \rep(C_r) \to \rep(\sigma C_r)$ $($see \cite[2.3]{Ri6}$)$. By the same token we have a functor $\Phi^- \colon \rep(\sigma C_r) \to \rep(C_r)$ given by the composition of the reflection functors for all the sources of $\sigma C_r$. Then $\mathcal{F} := \Phi^- \circ \Phi^+ \colon \rep(C_r) \to \rep(C_r)$ satisfies $\cF(M) \cong \tau^{-1}_{C_r} M$ for $M \in \rep(C_r)$  indecomposable and non-injective \cite[2.3]{Ri6},\cite[VII 5.8.]{Assem1}. Therefore statements $(a)$ and $(b)$ of the next Lemma follow immediately from the definition of the reflection functors. 
  
\begin{Lemma}\label{LemmaLeavesExistence}
Let $M$ be in $\rep(C_r)$ indecomposable and not injective. 
\begin{enumerate}[topsep=0em, itemsep= -0em, parsep = 0 em, label=$(\alph*)$]  
\item For each $x \in C_r^+$ we have $\dim_k (\tau^{-1}_{C_r} M)_x = (\sum_{y \in x^+} \dim_k M_y) - \dim_k M_x.$
\item For each $y \in C_r^-$ we have $\dim_k (\tau^{-1}_{C_r} M)_y = (\sum_{x \in y^-} \dim_k (\tau^{-1}_{C_r} M)_x) - \dim_k M_y.$
\item Let $0 \to A \to B \to C \to 0$ be an almost split sequence with $B$ indecomposable. If $a \in C^-_r$ is a leaf of $A$, then $B$ has a leaf in $C_r^-$.
\item Let $0 \to A \to B \to C \to 0$ be an almost split sequence with $B$ indecomposable. If $a \in C_r^+$ is a leaf of $A$ and $b \in a^+$ satisfies $\dim_k A_b = \dim_k A_a$, then $a$ is a leaf of $B$.
\end{enumerate}
\end{Lemma}
\begin{proof}

$(c)$ Consider a path $a \leftarrow b \to c$ such that $b,c$ are not in $\supp(A)$ as illustrated in Figure $\ref{Fig:ProofFigure1}$. Since $b$ is in $C_r^+$ we get with $(a)$ that

\[\dim_k  C_b = \dim_k (\tau^{-1}_{C_r} A)_b = (\sum_{y \in b^+} \dim_k A_y) - \dim_k A_b = \dim_k A_a - \dim_k A_b = \dim_k A_a \neq 0.\]
Now let $d \in n(c) \setminus \{b\}$. Then $\dim_k C_d =  (\sum_{y \in d^+} \dim_k A_y) - \dim_k A_d = 0$. Hence we get that
$\dim_k  C_c = \dim_k (\tau^{-1}_{C_r} A)_c \stackrel{(b)}{=} (\sum_{x \in c^-} \dim_k (\tau^{-1}_{C_r} A)_x) - \dim_k A_c = \dim_k C_b - \dim_k A_c = \dim_k C_b \neq 0.$
Hence $c \in \supp(C)$ is a leaf of $C$ and since $\supp(B) = \supp(A) \cup \supp(C)$ we conclude that $c \in C_r^-$ is a leaf of $B$.\\
\begin{figure}[!h]
%\centering 
\includegraphics[width=0.35\textwidth, height=55px]{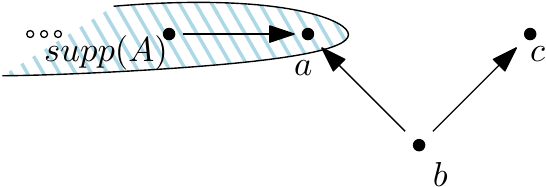}
\caption{Illustration of the setup for $(c)$.}
\label{Fig:ProofFigure1}
\end{figure}

\noindent $(d)$ Application of $(a)$ yields $($see Figure $\ref{Fig:ProofFigure2})$
\[ \dim_k C_a = \dim_k (\tau_{C_r}^{-1} A)_a = (\sum_{z \in a^+}\dim_k A_z) - \dim_k A_a = \dim_k A_b - \dim_k A_a = 0.\]
Now fix $c \in a^{+} \setminus \{b\}$, then $c \in C_r^-$. Let $d \in c^- \setminus \{a\}$ then $\dim_k A_f = 0$ for all $f \in d^+ \cup \{d\}$, since $f \notin \supp(A)$.
Hence we get
\[\dim_k  C_d = \dim_k (\tau^{-1}_{C_r} A)_d \stackrel{(a)}{=} (\sum_{y \in d^+} \dim_k A_y) - \dim_k A_d = 0.\]
We conclude
\[ \dim_k C_c = \dim_k (\tau_{C_r}^{-1} A)_c \stackrel{(b)}{=} (\sum_{z \in c^-} \dim_k (\tau^{-1}_{C_r} A)_z) - \dim_k A_c = (\sum_{z \in c^-} \dim_k C_z) - \dim_k C_c = 0 - 0 = 0.\]
We have shown that $(a^+ \setminus \{b\}) \cap (\supp(A) \cup \supp(C)) = \emptyset$. Since 
$\supp(A) \cup \supp(C) = \supp(B)$ we get $|\supp(B)  \cap n(a)| = |(\supp(A) \cup \supp(C) ) \cap n(a)| \leq |\{b\}| = 1$. Since $a \in \supp(A) \subseteq \supp(B)$ the vertex $a \in C_r^+$ is a leaf of $T(\supp(B))$. Since $B$ is indecomposable we have $T(B) = T(\supp(B))$.  

\begin{figure}[!h]
\includegraphics[width=0.32\textwidth, height=60px]{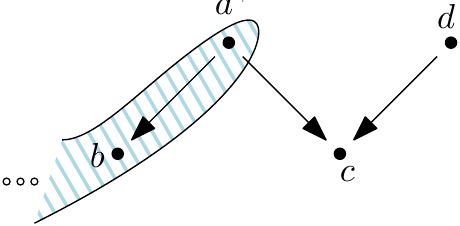}
\caption{Illustration of the setup for  $(d)$.}
\label{Fig:ProofFigure2}
\end{figure}      

\end{proof}

Let $M_1 \in \rep(C_r)$ be regular with $\dim_k M_1 = 2$. We  define inductively a sequence of indecomposable representations in the regular component $\overline{\cD}$ of $M_1$. The representation $M_1$ is quasi-simple. Assume that $M_n$ is already defined. If $n$ is odd, then $M_{n+1}$ is the unique indecomposable
representation with irreducible epimorphism $M_{n+1} \to M_n$; if $n$ is even, then $M_{n+1}$ is the unique
indecomposable representation with irreducible monomorphism $M_n \to M_{n+1}$. The component $\overline{\cD}$ is shown in Figure $\ref{Fig:Component_X^1}$. We have $\cW_C(\overline{\cD}) = \cW(\pi_{\lambda}(\overline{\cD})) = 1$ \cite[Example 3.3]{Wor1}.

\begin{figure*}[!h]
\centering 
\includegraphics[width=0.5\textwidth, height=150px]{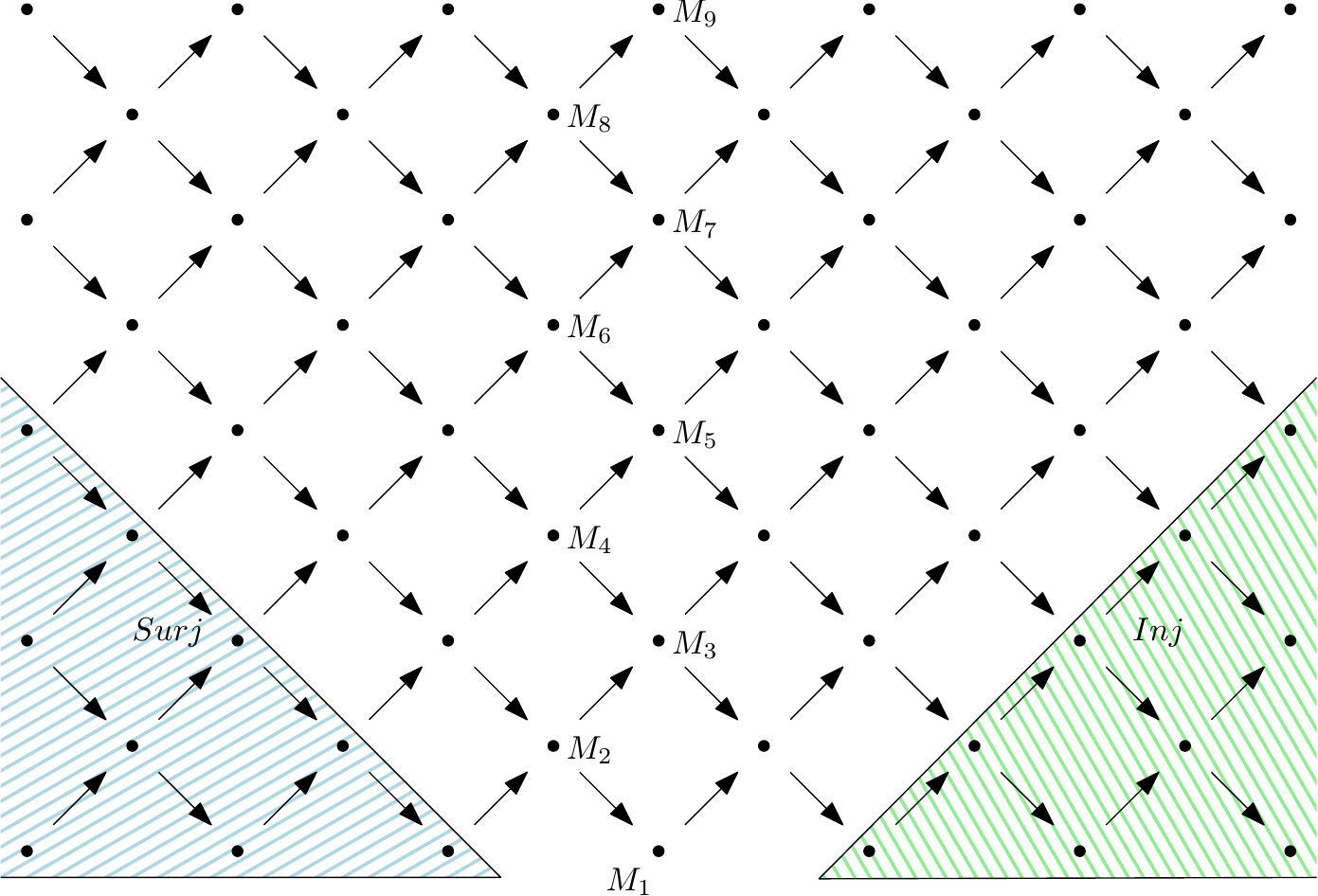}
\caption{Component containing the family $(M_n)_{n \geq 1}$.}
\label{Fig:Component_X^1}
\end{figure*}

\begin{Theorem}\label{TheoremMainTheorem2}
Let $r \geq 3$.
\begin{enumerate}[topsep=0em, itemsep= -0em, parsep = 0 em, label=$(\alph*)$]  
\item For each $m \geq 1$ there is $n_0 \geq 1$ and a family $(\cD_n)_{n \geq n_0}$ of regular components with $\cW_C(\cD_n) = m$ and $\cD_n$ contains a quasi-simple representation $E_n$ with $\dimu \pi_{\lambda}(E_n) = (n+1,n)$.
\item $\NN \subseteq \{ \cW_C(\cE) \mid \cE \in \cR(C_r) \} \subseteq \NN_0$.
\item $\{ \cW(\cC) \mid \cC \in \cR \} = \NN_0$.
\end{enumerate}
\end{Theorem}
\begin{proof}
Recall that $\cW_{C}(\overline{\cD}) = 1$ and $\cW_{C}(\overline{E}) = 2$ for the regular component $\overline{E}$ containing $X_1$. Fix $l \geq 1$. Then  we have $D \pi_{\lambda}(M_{2l+1}) \cong \pi_{\lambda}(M_{2l+1})$ by \cite[Example 3.3]{Wor1}. By $\ref{LemmaDual}$ $M_{2l+1}$ is balanced. Moreover we have $d^-(M_{2l+1}), d^+(M_{2l+1}) \geq 2$. Hence $\ref{TheoremMainTheorem}$ yields $n_0 \in \NN$ and an infinite family of components $(\cD_n)_{n\geq n_0}$ of width $\cW_{C}(\cD_n) = \cW_C(\overline{\cD}) + \ql(M_{2l+1})- 1 = 2l + 1$. By Corollary $\ref{CorollarySpecialmodule}$ we find a balanced and quasi-simple representation $A \in \cD_n$ that satisfies the assumption of $\ref{LemmaLeavesExistence}(d)$. Consider the AR sequence $0 \to A \to B \to C \to 0$. Then $B$ is balanced by $\ref{LemmaLeavesExistence}(c),(d)$, $2 \leq d^+(B),d^-(B)$ and $\ql(B) = 2$. By $\ref{TheoremMainTheorem}$ we get an infinite family of components of width $\cW_C(\cD_n) + \ql(B) - 1 =  (2l+1) + 2 - 1 = 2l + 2$. This proves $(a)$ and $(b)$. For $(c)$ observe that there exist regular components $\cC \in \cR$ with $\cW(\cC) = 0$ \cite[3.3]{Wor1}.
\end{proof}

\subsection{Counting regular components of fixed width}

This section is motivated by the following result by Kerner and Lukas.

\begin{PropositionN} \cite[5.2]{Ker2} \label{CorollaryKernerUncountable} Assume that $k$ is uncountable and $A$ is a wild hereditary algebra with $n > 2$ simple modules. Then the number of regular component of $A$ with quasi-rank $-1$ is uncountable. Moreover the set of components of quasi-rank $\leq -1$ for the Kronecker algebra is  uncountable.
\end{PropositionN}

The proof of the second statement uses the first statement for the path algebra $A$ of the wild quiver $1 \rightarrow 2 \rightrightarrows 3$ with $n(A) = 3$ simple modules and the existence of a regular tilting module $T_r$ in $\modd A$ that induces a bijection
\[ \varphi \colon \{ \cC \mid \cC \ \text{regular component of} \ A \} \to \cR \]
with $\rk(\varphi(\cC)) \leq \rk(\cC)$ for all $\cC \in \{ \cD \mid \cD \ \text{regular component of} \ A \}$. To generalize the arguments from $\leq -1$ to $\leq -p$ for $p \in \NN$ one would need the existence of bricks of arbitrary quasi-length. Unfortunately for each  hereditary algebra there is a finite upper bound for the quasi-length of regular bricks given by the number of simple modules $-1$, which is in our case $n(A) - 1 = 2$. We show how to circumvent this obstacle by considering an action of the general linear group $\GL_r(k)$ on $\rep(\Gamma_r)$.

\begin{Definition}\cite[3.6]{CFP1}
Denote with $\GL_{r}(k)$ the group of invertible $r \times r$-matrices with coefficients in k which acts on $\bigoplus^r_{i=1} k\gamma_i$ via $A.\gamma_j = \sum^r_{i=1} a_{ij} \gamma_i$ for $1 \leq j \leq r$, $A \in \GL_{r}(k)$. For $A \in \GL_r(k)$, let $\varphi_A \colon \cK_r \to \cK_r$ the algebra homomorphism with $\varphi_A(e_1) = e_1$, $\varphi_A(e_2) = e_2$ and $\varphi_A(\gamma_i) = A.\gamma_i$, $1 \leq i \leq r$. For a $\cK_r$-module $M$ denote the pullback of $M$ along $\varphi_{A^{-1}}$ by $A.M$. The module $M$ is called $\GL_r(k)$-stable if $A.M \cong M$ for all $g \in \GL_r(k)$, in other words if $\GL_r = \GL_r(k)_M := \{ A \in \GL_r(k) \mid A.M \cong M\}$.
\end{Definition}

\begin{examples}
\begin{enumerate}[topsep=0em, itemsep= -0em, parsep = 0 em, label=$(\alph*)$]
\item  The simple representations of $\Gamma_r$ are $\GL_r(k)$-stable and by \cite[2.2]{Far1} every preinjective and every preprojective representation is $\GL_r(k)$-stable.
\item There are  $\GL_r(k)$-stable  representations that are regular \cite[1.2]{Wor1}. In this case all representations in the same component are also $\GL_r(k)$-stable.
\item Recall that the preinjective representation $I_3 = \tau I_1$ has dimension vector $(3r-1,r)$. Let $M$ be in $\rep(C_r)$ with $\pi_{\lambda}(M) \cong I_3$. The support of $M$ for $r = 3$ is shown in Figure $\ref{Fig:SupportSymmetric}$. Let $c \in T(M)_0$ be the vertex $\dim_k M_c = 2$.

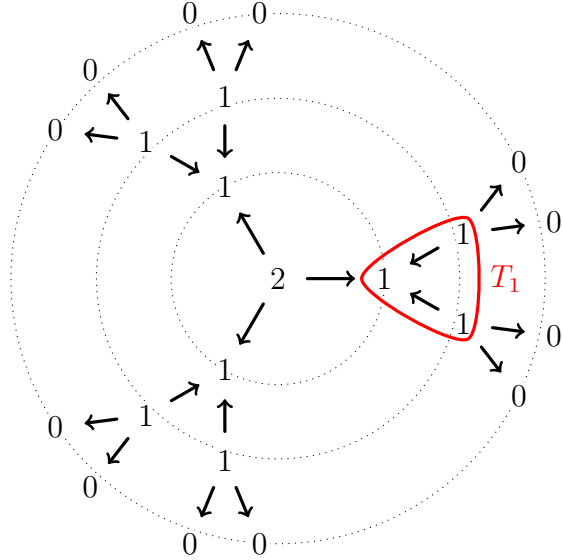
\begin{figure}[!h]
\centering 

\tikzstyle{level 1}=[sibling angle=120,  level distance=1.4cm]
\tikzstyle{level 2}=[sibling angle=60,level distance=1.2cm]
\tikzstyle{level 3}=[sibling angle=45]
\tikzstyle{every node}=[]
\tikzstyle{edge from child}=[]
\begin{tikzpicture}[grow cyclic,shape=circle,very thick,level distance=15mm,
                    cap=round]
                    \node {2} child [color=black] foreach \A in {1,1,1}
    { node {\contour{white} {\A}} edge from parent[->] child [color=black] foreach \B in {1,1}
        {  node {\contour{white} {\B}} edge from parent[<-] child [color=black] foreach \C in {0,0}
            { node {\contour{white} {\C}} edge from parent[->]} 
        }
    };
    \begin{scope}[on background layer]
    \draw[black,dotted] (0,0) circle (40pt);
    \draw[black,dotted] (0,0) circle (68pt);
    \draw[black,dotted] (0,0) circle (100pt);
    \end{scope}

\draw [red] plot [smooth cycle] coordinates {(1.1,0) (2.5,0.8) (2.5,-0.8)};
 \node [red]at (3,0) {$T_1$}; 
\end{tikzpicture}

\caption{Support of $M$ with $\pi_{\lambda}(M) \cong I_3$ for $r = 3$.}
\label{Fig:SupportSymmetric} 
\end{figure}

The underlying tree of $\supp(M)$ is symmetric in the following sense. The quiver $T(M)  \setminus \{c\}$ is not connected and consists of $r=3$ isomorphic trees $T_1,T_2,T_3$. Moreover for $n \in \{1,2\}$ the sum  $\dim_k M_x$ for $x \in (T_i)_0$ with distance $d(c,x) = n$ is independent of $i \in \{1,2,3\}$.  We will prove that this is not a coincidence. We   show that every representation $M$ such that $\pi_{\lambda}(M)$ is $\GL_r(k)$-stable, has a central point.
\end{enumerate}
 
\end{examples}

\subsubsection{$\cS_r$-stability}

Denote by $\cS_r$ the symmetric group on $\{1,\ldots,r\}$.
%Assume that $M \in \rep(Q)$ is indecomposable such that $\pi_{\lambda}(M)$ is $\GL_r$-stable. 
Then each for each $\sigma \in \cS_r$ there is an induced bijection on $\{1,\ldots,r\} \to (\Gamma_r)_1$ given by $i \mapsto \gamma_{\sigma(i)}$ which extends in a natural way to the set of equivalence classes of walks in $\Gamma_r$. By abuse of notation we denote by $\sigma \colon C_r \to C_r$ the induced quiver automorphism. 
Let $\alpha \colon [w] \to [\gamma_i w]$ be an arrow in $C_r$, then by definition $\sigma(\alpha)$ is the unique arrow 
 $\sigma(\alpha) \colon \sigma([w]) \to [\gamma_{\sigma(i)}] \sigma([w])$. Note that $\pi(\alpha) = \gamma_i$ and $\pi(\sigma(\alpha)) = \gamma_{\sigma(i)}$.\\
Now let $M \in \rep(C_r)$ be an indecomposable representation. We define $\sigma(M)$ to be the indecomposable representation with 
\[ \sigma(M)_x := M_{\sigma(x)} \ \text{and} \ \sigma(M)(\alpha) := M(\sigma(\alpha)).\] 
We say that $M$ is $\cS_r$-stable if for each $\sigma \in \cS_r$ there is $g_{\sigma}$ with $M \cong \sigma(M)^{g_\sigma}$.
This definition is motivated by the following obvious result:

\begin{corollary}
Let $M \in \rep(C_r)$ be an indecomposable representation. If $\pi_{\lambda}(M)$ is $\GL_r(k)$-stable then $M$ is $\cS_r$-stable.
\end{corollary}
\begin{proof}
We let $I(\sigma)$ be the permutation matrix given by $\sigma$, i.e.
$I(\sigma)_{ij} = 1$ if and only if $\sigma(i) = j$ and $I(\sigma)_{ij} = 0$ otherwise.
Now we assume that $\pi_{\lambda}(M)$ is $\GL_r(k)$-stable. Then we get for each $\sigma \in \cS_r$ that 
\[ \pi_{\lambda}(\sigma(M)) = I(\sigma).\pi_{\lambda}(M) \cong \pi_{\lambda}(M).\]
Hence we find $g_\sigma \in G$ such that 
\[ M \cong \sigma(M)^{g_\sigma}. \]

\end{proof}

\noindent Note that $\sigma([\epsilon_1]) = [\epsilon_1]$ and since $G$ acts freely $C_r$, the element $g_\sigma$ is uniquely determined. In the following we study the quiver automorphisms $\sigma \circ g_\sigma\colon C_r \to C_r$.

\subsubsection{Automorphisms of trees}

A group G is said to have property FA \cite[I.6.1]{Se1} if every action of G  on a tree $T$ by graph automorphisms (which do not invert an edge) has a global fixed point $z \in T_0$, i.e. $gz = z$ for all $g \in G$. It is known \cite[I.6.3.1]{Se1} that all finitely generated torsion groups have the property FA. In particular, for every group action of a finite group acting on  a quiver which underlying graph is a tree, there is a global fixed point.

\begin{Definition}
For $x \in C_r^+$ and $1 \leq i \leq r$ denote by $T(x,i)$ the connected component of $C_r \setminus \{x\}$ containing $t(\alpha_i)$, where $\alpha_i \colon x \to t(\alpha_i)$ is the unique arrow with $\pi(\alpha_i) = \gamma_i$. \\
Let $M$ in $\rep(C_r)$ be indecomposable and $x \in \supp(M)$, then we define $T(x,i,M) := T(M) \cap T(x,i)$.
\end{Definition}
 
\noindent Note that $\supp(M) = \{x\} \cup \bigcup^{r}_{i=1} T(x,i,M)_0$.

\begin{proposition}\label{PropositionCenter}
Let $M \in \rep(C_r)$ be $\cS_r$-stable. Then there is $c \in \supp(M)$ such that 
\begin{enumerate}[topsep=0em, itemsep= -0em, parsep = 0 em, label=$(\alph*)$]
\item $\sigma \circ g_\sigma(c) = c$ for all $\sigma \in \cS_r$.
%\item $|n(c) \cap \supp(M)| = r$.
\item For each $n \in \NN$ the number $r$ divides $D(n,c) := \sum_{x \in \supp(M),d(x,c) = n} M_x$. 
\end{enumerate}
\end{proposition}
\begin{proof}
Since $M \cong \sigma(M)^{g_\sigma}$, we have 
\[\supp(M) = \supp(\sigma(M)^{g_\sigma}) = \{g^{-1}_\sigma.x \mid x \in \supp(\sigma(M)) \} = \{g^{-1}_\sigma \circ \sigma^{-1}(x) \mid x \in \supp(M) \}.\]
 We assume that $\dim_k M \neq 1$, otherwise there is nothing to show. Let $T \subseteq C_r$ be the finite subtree $T := T(M)$. Then $\sigma \circ g_\sigma  \colon T \to T$ is a quiver automorphism of $T$. Since $T$ is finite, $\Aut(T)$ is finite and there exists a vertex $c \in T_0$ with $\varphi(c) = c$ for all $\varphi \in \Aut(T)$. We assume that $c \in C_r^+$.
For $1 \leq i \leq r$ we let $\beta_i \colon c \to t(\beta_i)$ be the unique arrow with $\pi(\beta_i) = \gamma_i$ and set $T_i := T(c,i,M)$. Since $\dim_k M \neq 1$ and $M$ is indecomposable, we can assume w.l.o.g. that $e:= t(\beta_1) \in \supp(M)$. Since every automorphism of $C_r$ respects the metric $($see $\ref{DefinitionMetric})$ we get 
\[ 1 = d(c,e) = d(\sigma \circ g_\sigma(c),\sigma \circ g_\sigma (e)) = d(c,\sigma \circ g_\sigma (e)).\]
In particular, $\sigma \circ g_\sigma (\beta_1) \in \{\beta_1,\ldots,\beta_r\}$. Now fix $j \in \{1,\ldots,r\} \setminus\{1\}$ and $\sigma := (1 \ j) \in \cS_r$. Then we have $\pi(\sigma(\beta_1)) = \gamma_j$. Since $\pi \circ g = \pi$ for all $g \in G$, we get $\sigma \circ g_\sigma(\beta_1) = \beta_j$ and conclude $\sigma \circ g_\sigma(T_1) = T_j$. Hence $T_1,\ldots
,T_r$ are non-empty isomorphic quivers. For each $n \in \NN$ and $i \in \{1,\ldots,r\}$ we define $d_{n,i,c} := \{ x \in (T_i)_0 \mid d(x,c) = n\}$.  Let now $x \in d_{n,1,c}$, then we have $\sigma \circ g_\sigma  (x) \in d_{n,j,c}$ since $\sigma \circ g_\sigma(x) \in T_j$ and   
\[ n = d(c,x) = d(\sigma \circ g_\sigma(c),\sigma \circ g_\sigma (x)) = d(c,\sigma \circ g_\sigma (x)).\]
%Observe that $(g_\sigma \circ \sigma)^2$ is the identity on $T$ and therefore  $g_{\sigma} \circ \sigma = \sigma \circ g_{\sigma}$.
Moreover we have $M_x = (\sigma(M)^{g_\sigma})_x = M_{\sigma \circ g_{\sigma}(x)}$.
It follows $\sum_{y \in d_{n,1,c}} \dim_k M_y = \sum_{y \in d_{n,j,c}} \dim_k M_y$ and we conclude 
\[ D(n,c)  = \sum_{x \in \supp(M),d(x,c) = n} \dim_k M_x = r \cdot \sum_{y \in d_{n,1,c}} \dim_k M_y.\]
\end{proof}

\begin{corollary}\label{CorollaryNotGradable}
Assume that $M$ is $\cS_r$-stable. If $\dim \pi_{\lambda}(M) = (a,b)$ then $r$ divides $a$ or $b$.
\end{corollary}
\begin{proof}
Let $c \in \supp(M)$ be as in $\ref{PropositionCenter}$. Then we have 
\[ \dim_k M = \dim_k M_c + \sum_{n \in 2\NN-1} D(n,c)+  \sum_{n \in 2\NN} D(n,c) .\]
If $c \in C^+_r$, then $b = \sum_{n \in 2\NN-1} D(n,c)$ and $a = \sum_{n \in 2\NN-1} D(n,c)$ otherwise. Hence $r$ divides $b$ or $a$.
\end{proof}

\noindent As an application we get the following result for components of the Kronecker quiver $\Gamma_r$.

\begin{corollary}\label{LemmaExistenceNoStable}
Let $m \in \NN$, then there exists a regular component $\cC$ with $\cW(\cC) = m$ and no representation in $\cC$ is $\GL_r(k)$-stable.
\end{corollary}
\begin{proof}
Let $m \geq 1$. By $\ref{TheoremMainTheorem}$ there exists $n_0 \in \NN$ such that for each $n \geq n_0$ there is a regular component $\cD_n$ with $\cW(\cD_n) = m$ and $\cD_m$ contains a quasi-simple representation $E_n = \pi_{\lambda}(F_n)$ with $F_n \in \rep(C_r)$ and $\dimu E_n = (n+1,n)$. Since $r \geq 3$, we find $l \geq n_0$ $($even infinitely many$)$ such that $r$ does not divide $l$ and $l+1$. Hence $F_l$ is not $\cS_r$-stable and $E_l$ not $\GL_r(k)$-stable. Therefore no representation in $\cD_l$ is $\GL_r(k)$-stable by \cite[2.2]{Far1}.
%The representation $\pi_{\lambda}(X^1)$ has dimension vector $(1,r-1)$ and is contained in a component of width $2$. Let $E \in \rep(Q)$ an indecomposable representation with $\dimu \pi_{\lambda}(E) = (1,1)$. Then $\pi_{\lambda}(E)$ is contained in a component of width $1$.
\end{proof}

\subsubsection{The number of regular components in $\rep(\Gamma_r)$}

\begin{Definition}
A locally closed set is an open subset of a closed set. A constructible set is a finite union of locally closed sets.
\end{Definition}

%\begin{Theorem}
%If $f \colon X \to Y$ is a morphism of algebraic varieties, then $f(X)$ is constructible.
%\end{Theorem}

\begin{Lemma}\label{LemmaInjection}
Let $M \in \rep(\Gamma_r)$ with $\GL_r(k)_M \neq \GL_r(k)$. There is an injection $\iota \colon k \to \GL_r(k)/\GL_r(k)_M$.
\end{Lemma}

\begin{proof}
By \cite[2.1]{Far1} $\GL_r(k)_M$ is a closed subgroup of $\GL_r(k)$ and by \cite[5.5]{Mall} an $\GL_r(k)/\GL_r(k)_M$ an algebraic variety. Hence we find an affine variety $V \subseteq \GL_r(k)/\GL_r(k)_M$ with $d := \dim V = \dim \GL_r(k)/\GL_r(k)_M $. Since $\GL_r(k)$ is irreducible we have $\dim \GL_r(k)/\GL_r(k)_M = \dim \GL_r(k) - \dim \GL_r(k)_M \geq 1$.   Let $k[t_1,\ldots,t_d] \to k[V]$ be a Noether-normalization and $\varphi^\ast \colon V \to \A^d$ be the comorphism. Then $\varphi^\ast$ is dominant. Hence there is a dominant morphism $f \colon V \to \A^1$. By Chevalley's Theorem $f(V)$ is constructible and hence finite or cofinite. Since $f(V)$ is dense in $\A$, $f(V)$ is not finite and therefore cofinite. That means $|k \setminus C|$ is finite. Since $k$ is  infinite we have $|f(V)|= |k|$. It follows
$|k| = |\A| = |f(V)| \leq |V| \leq |\GL_r(M)/\GL_r(k)_M|.$
\end{proof}

\begin{Theorem}\label{TheoremMainTheorem3}
Let $m \in \NN$. There is a bijection $\{ \cC \in \cR \mid \cW(\cC) = m\} \to k$.
\end{Theorem}
\begin{proof}
It is well known that $|\cR| = |k|$, see \cite[XVIII 1.8]{Assem3}. In particular there is an injection 
\[\{ \cC \in \cR \mid \cW(\cC) = m\} \to k.\] 
By $\ref{LemmaExistenceNoStable}$ there is a regular component $\cC$ with $\cW(\cC) = m \in \NN$ such that no representation in $\cC$ is $\GL_r(k)$-stable. For $E$ in $\cC$ the map
\[ \GL_r(k) / {\GL_r(k)_E} \to \{ Z \in \rep(\Gamma_r) \mid \dimu Z = \dimu E \}; A\GL_r(k)_E \mapsto A.E \]
 is well defined and injective. Since the number of representations in a regular component with given dimension vector $(a,b)$ is $\leq 1$     \cite[XIII.1.7]{Assem3} and $\GL_r(k)$ acts via auto equivalances we get with $\ref{LemmaInjection}$ an injection

\[ k \to \GL_r(k) / \GL_r(k)_E \to \{ Z \in \rep(\Gamma_r) \mid \dimu Z = \dimu E \}  \to \{ \cC \in \cR \mid \cW(\cC) = m\}. \]
By the Schr\"oder-Bernstein Theorem we get the desired bijection.
\end{proof}

\begin{Remark}
Note that we restrict ourselves to components of width $\geq 1$, since we dont know whether components in $\rep(C_r)$ of width $0$ exist. Also the examples \cite[3.3]{Wor1} of components of width $0$ in $\rep(\Gamma_r)$ are $\GL_r(k)$-stable. For the case $n = 0$ we argue as follows.

\end{Remark}

\begin{Lemma}
Let $r \geq 3$, then there exists an bijection $k \to \{ \cC \in \cR \mid \cW(\cC) = 0 \}$.
\end{Lemma}
\begin{proof} The proof of \cite[3.3.3]{Bi1} yields an injective map $\varphi \colon \ind \cE(\cX) \to \{ \cC \in \cR \mid \cW(\cC) = 0 \}$, where $\ind \cE(\cX)$ are the indecomposable objects in a category $\cE(\cX)$ equivalent to the category of finite dimensional modules over the power-series ring $k\langle \langle X_1,\ldots,X_t \rangle \rangle$ in non-commuting variables $X_1,\ldots,X_t$ and $t \geq 2$. Now let $\lambda \in k \setminus \{0\}$ and consider the indecomposable module $M_\lambda = k^2$ given by $X_1.(a,b) =  (\lambda b,0)$, $X_2.(a,b) = (b,0)$ and $X_i.(a,b) = 0$ for $i > 2$. Then $M_\lambda \not\cong M_\mu$ for $\lambda \neq \mu$ and we have an injection $k \to \ind \cE(\cX)$. The claim follows as in $\ref{TheoremMainTheorem3}$.
\end{proof}

\begin{corollary}
Let $r \geq 3$, then for each $n \in \NN$ there are exactly $|k|$ regular components such that $\rk(\cC) \in [-n,-n+3]$.
\end{corollary}

\begin{corollary} Assume that $k$ is uncountable and $p \in \NN$. The set of components of quasi-rank $\leq -p$ in $\cR$ is  uncountable.
\end{corollary}

%%%%%%%%%%%%%%%%%%%% Acknowledgement %%%%%%%%%%%%%%%%%%

\section*{Acknowledgement}
The results of this article are part of my doctoral thesis, which I am currently
writing at the University of Kiel. I would like to thank my advisor Rolf Farnsteiner for his continuous support and helpful comments.
I also would like to thank the whole research team for the very pleasant working atmosphere and the encouragement throughout my studies. In particular, I thank Christian Drenkhahn for proofreading.\\
Furthermore, I thank Claus Michael Ringel for fruitful discussions during my visits in Bielefeld.

%%%%%%%%%%%%%%%%%%%%%%% REFERENCES %%%%%%%%%%%%%%%%%%%


\begin{bibdiv}
\begin{biblist}


\bib{Assem1}{book}{
title={Elements of the Representation Theory of Associative Algebras, I}
subtitle={Techniques of Representation Theory}
series={London Mathematical Society Student Texts}
author={I. Assem},
author={D. Simson},
author={A. Skowro\'nski},
publisher={Cambridge University Press},
date={2006},
address={Cambridge}
}


\bib{Assem3}{book}{
title={Elements of the Representation Theory of Associative Algebras, III},
subtitle={Representation-Infinite Tilted Algebras},
series={London Mathematical Society Student Texts}
author={I. Assem},
author={D. Simson},
author={A. Skowro\'nski},
publisher={Cambridge University Press},
date={2007},
address={Cambridge}
}


\bib{Bi1}{article}{
title={Representations of constant socle rank for the Kronecker algebra}
author={D. Bissinger},
status={Preprint, arXiv:1610.01377v1},
date={2016}
}


\bib{CFP1}{article}{
title={Representations of elementary abelian $p$-groups and bundles of Grassmannians},
author={Carlson, J. F.},
author={Friedlander, E. M.},
author={Pevtsova, J.}
date={2012},
journal={Advances in Mathematics},
volume={229},
pages={2985-3051}
}
%
%\bib{CFS1}{article}{
%title={Modules for $\ZZ_p \times \ZZ_p$},
%author={J. F. Carlson},
%author={E. M. Friedlander},
%author={A. Suslin},
%journal={Commentarii Math. Helv.},
%date={2011},
%volume={86},
%pages={609-657}
%}
%




\bib{BoChen1}{article}{
title={Dimension vectors in regular components over wild
Kronecker quivers},
journal={Bulletin des Sciences Math\'{e}matiques},
volume={137},
pages={730-745},
author={B. Chen},
date={2013}
}

\bib{Far1}{unpublished}{
title={Categories of modules given by varieties of $p$-nilpotent operators},
status={Preprint, arXiv:1110.2706v1},
author={R. Farnsteiner},
date={2011}
}



%\bib{Far2}{unpublished}{
%author ={R. Farnsteiner},
%title={Lectures Notes: Nilpotent Operators, Categories of Modules, and Auslander-Reiten Theory},
%note={http://www.math.uni-kiel.de/algebra/de/farnsteiner/material/Shanghai-2012-Lectures.pdf},
%date={2012}
%}



%\bib{Far3}{article}{
%title={Extensions of tame algebras and finite group schemes of domestic representation type},
%author={R. Farnsteiner},
%date={2015},
%journal={Forum Mathematicum (2)},
%volume={27},
%pages={1071-1100},
%
%}


\bib{Gab2}{article}{
title={Covering spaces in representation theory},
author={K. Bongartz},
author={P. Gabriel},
journal={Inventiones mathematicae},	
year={1981/82},
pages = {331-378},
volume = {65},
%address={Paris}
}

\bib{Gab3}{article}{
title={The universal cover of a representation finite algebra. Representations of algebras},
journal={Lecture Notes in Mathematics},
volume={903},
pages={68-105},
date={1981},
author={P. Gabriel},
%address={Paris}
}



\bib{Ker1}{article}{
title={Exceptional Components of Wild Hereditary Algebras},
author={O. Kerner},
journal={Journal of Algebra},
volume={152},
pages={184-206},
number={1},
date={1992}
}

\bib{Ker2}{article}{
title={More Representations of Wild Quivers},
author={O. Kerner},
journal={Contemporary Mathematics},
series={Expository Lectures on Representation Theory}
volume={607},
date={2014},
pages={35-66}
}

\bib{Ker3}{article}{
title={Representations of Wild Quivers},
journal={Representation theory of algebras and related topics, CMS Conf. Proc.},
volume={19},
date={1996},
pages={65-107}, 
author={O. Kerner},
}

\bib{KerLuk1}{article}{
title={Elementary modules},
author={O. Kerner},
author={F. Lukas},
journal={Math. Z.},
volume={223},
pages={421-434},
date={1996}
}

\bib{KerLuk2}{article}{
title={Regular modules over wild hereditary algebras},
author={O. Kerner},
author={F. Lukas},
journal={in Proc. Conf. ICRA '90, CMS Conf. Proc.},
volume={11},
pages={191-208},
date={1991}
}



\bib{Mall}{book}{
title={Linear Algebraic Groups and Finite Groups of Lie Type},
series={Cambridge Studies in Advanced Mathematics, 133},
author={G. Malle},
author={D. Testerman}
date={2011}
publisher={Cambridge University Press},
adress={Cambridge}
%address={Paris}
}

\bib{Ried1}{article}{
title={Algebren, Darstellungsk\"ocher, Ueberlagerungen und zur\"uck},
author={C. Riedtmann},
journal={Commentarii Math. Helv.},
volume={55},
pages={199-224},
date={1980}
}

\bib{Ri3}{article}{
title={Finite-dimensional hereditary algebras of wild representation type},
author={C.M. Ringel},
journal={Math. Z.},
volume={161},
pages={235-255},
date={1978}
}

\bib{Ri4}{article}{
title={Representations of $K$-species and bimodules},
author={C.M. Ringel},
date={1976},
journal={Journal of Algebra},
volume={41},
number={2},
pages={269-302}
}

\bib{Ri5}{book}{
title={Tame Algebras and Integral Quadratic Forms},
author={C.M. Ringel},
date={1984},
publisher={Springer Verlag}
volume={1099},
series={ Lecture Notes in Mathematics}
pages={269-302}
}



\bib{Ri6}{article}{
title={Indecomposable representations of the Kronecker quivers},
author={C.M. Ringel},
journal={Proc. Amer. Math. Soc.},
volume={141},
date={2013},
number={1}
%address={Paris}
}




\bib{Ri7}{webpage}{
title={Covering Theory},
author={C.M. Ringel},
url={https://www.math.uni-bielefeld.de/~ringel/lectures/izmir/izmir-6.pdf},
%address={Paris}
}

%\bib{Ri8}{article}{
%title={The elementary $3$-Kronecker modules},
%author={C.M. Ringel},
%%address={Paris}
%status={Preprint, arXiv:1612.09141},
%date={2016}
%}

\bib{Se1}{book}{
title={Trees},
series={Springer Monographs in Mathematics}
author={J-P. Serre},
date={1980},
publisher={Springer-Verlag},
address={Berlin},
}



\bib{Wor1}{article}{
title={Categories of modules for elementary abelian p-groups and generalized Beilinson algebras},
author={J. Worch},
journal={J. London Math. Soc.},
volume={88},
date={2013},
pages={649-668}
}



\bib{Wor3}{webpage}{
author ={J. Worch},
title={Module categories and Auslander-Reiten theory for generalized Beilinson algebras. PhD-Thesis},
url={http://macau.uni-kiel.de/receive/dissertation_diss_00013419}
date={2013}
}


\end{biblist}
\end{bibdiv}
\end{document}